\documentclass{article}
\usepackage[utf8]{inputenc}

\usepackage{tgpagella}
\usepackage{amsmath,amsthm, amssymb, amsfonts}
\usepackage{todonotes}
\usepackage[symbol]{footmisc}
\usepackage{wrapfig}
\usepackage{paralist}
\usepackage{wasysym}
\usepackage{verbatim}
\usepackage{multirow}
\usepackage[]{algorithm2e}
\usepackage{comment}
\usepackage{fullpage}
\usepackage{mathtools}
\usepackage[pdfencoding=auto]{hyperref}
\usepackage{graphicx}
\usepackage{wrapfig}
\usepackage{todonotes}
\usepackage{mathrsfs}
\usepackage{enumerate}
\usepackage{dsfont}
\usepackage{units}
\usepackage{microtype}
\usepackage[]{algorithm2e}
\usepackage{ltablex}
\usepackage{varwidth}
\usepackage{pgfplots}
\usepackage{subfigure}
\usepackage{url}
\usepackage[subnum]{cases}

\usepackage{xcolor}

\allowdisplaybreaks[1]
\usepackage{sistyle} 
\newcommand\bSI[1]{{\small[\SI{}{#1}]}}

\makeatletter
\newlength\unitwdth
\newlength\numwdth
\newlength\tdima
\newcommand\SIdescr[2]{%
    \setlength\tdima{\linewidth}%
    \addtolength\tdima{\@totalleftmargin}%
    \addtolength\tdima{-\dimen\@curtab}%
    \addtolength\tdima{-\unitwdth}%
    \addtolength\tdima{-\numwdth}%
    \parbox[t]{\tdima}{%
        #1
        \leaders\hbox{$\m@th\mkern \@dotsep mu\hbox{\tiny.}\mkern \@dotsep mu$}%
        \hfill
        \ifhmode\strut\fi
        \makebox[0pt][l]{%
            \makebox[\unitwdth][l]{}%
            \makebox[\numwdth][r]{#2}}}}
\makeatother

\makeatletter
\newcommand{\mathleft}{\@fleqntrue\@mathmargin0pt}
\newcommand{\mathcenter}{\@fleqnfalse}
\makeatother

\newcommand{\N}{\mathbb{N}}

\newcommand{\R}{\mathbb{R}}
\newcommand{\loc}{\mathrm{loc}}

\newcommand{\bmat}[2]{\left[ \begin{array}{#1} #2 \end{array} \right]}
\newcommand{\dx}{\, \mathrm{d}x}
\newcommand{\ds}{\, \mathrm{d}s}
\newcommand{\dt}{\, \mathrm{d}t}
\newcommand{\dtau}{\, \mathrm{d}\tau}

 \DeclareMathOperator*{\esssup}{ess\,sup}

\newcommand{\cross}{\times}
\newcommand{\del}{\partial}
\newcommand{\grad}{\nabla}

\newcommand{\eps}{\varepsilon}
\newcommand{\Realization}{\mathrm{R}}

\newcommand{\sconc}{\odot}

\newcommand\restr[2]{{
  \left.\kern-\nulldelimiterspace 
  #1 
  \vphantom{\big|} 
  \right|_{#2} 
  }}

\newtheorem{theorem}{Theorem}[section]
\newtheorem*{theorem*}{Theorem}
\newtheorem{remark}[theorem]{Remark}
\newtheorem{definition}[theorem]{Definition}
\newtheorem{proposition}[theorem]{Proposition}

\newtheorem*{remark*}{Remark}
\newtheorem*{proposition*}{Proposition}

\numberwithin{equation}{section}

\title{Efficient Approximation of Solutions of Parametric Linear Transport Equations by ReLU DNNs}

\author{Fabian Laakmann\thanks{Mathematical Institute, University of Oxford, Andrew Wiles Building, Woodstock Road,  Oxford, OX2 6GG, UK, e-mail: \texttt{Fabian.Laakmann@maths.ox.ac.uk}} \and Philipp Petersen\thanks{Institut f\"ur Mathematik, Universit\"at Wien, Oskar-Morgenstern-Platz 1, 1090 Wien, Austria, e-mail: \texttt{Philipp.Petersen@univie.ac.at}}}

\date{ }

\begin{document}
\maketitle

\begin{abstract}
   We demonstrate that deep neural networks with the ReLU activation function can efficiently approximate the solutions of various types of parametric linear transport equations. For non-smooth initial conditions, the solutions of these PDEs are high-dimensional and non-smooth. Therefore, approximation of these functions suffers from a curse of dimension. We demonstrate that through their inherent compositionality deep neural networks can resolve the characteristic flow underlying the transport equations and thereby allow approximation rates independent of the parameter dimension.
\end{abstract}



\section{Introduction}

Linear parametric transport equations play an essential role in engineering, modelling, and mathematical physics where they describe physical phenomena of heat and mass transfer. 
A typical example is the transport of pollution in air or water depending on a set of parameters such as the direction and intensity of the flow of the fluid. 

In this work, we study to what extent the solutions of various types of parametric linear transport equations can be efficiently represented by deep neural networks. Concretely, we study variations of the following problem: Let $n,D,k \in \N$, and $T>0$. Let $V \in C^k([0,T]\cross \R^n \cross [0,1]^D; \R^n)$, $f \in C^k([0,T]\cross \R^n \cross [0,1]^D)$ and let $u_0 \in C^1(\R^n)$. We want to find $u \in C^1([0,T] \times \R^n \times [0,1]^D)$ such that 
\begin{align}\label{eq:Motivation}
\begin{cases}
\partial_t u(t,x,\eta) + V(t,x,\eta)\cdot \nabla_x u(t,x,\eta) = f(t,x,\eta),\\
u(0,x,\eta)=u_0(x). 
\end{cases}
\end{align}
The PDE of \eqref{eq:Motivation} has been studied extensively, see, e.g., \cite{Ambrosio2008, Ambrosio2005, Golse2000, GolseFrench} and we will recall the fundamentals in Section \ref{sec:linTransEq}. The set-up that we have in mind is one where the dimension of the parameter space $D$ is very high, $V$ is smooth, and $u_0$ is not very regular. Hence, direct approximation of $u$ amounts to \emph{approximating a high-dimensional function of low regularity}. In this formulation, the task is extremely challenging for classical methods.

While the global approximation problem is almost intractable, the method of characteristics shows that, even though $u$ is not smooth, its singularities revolve along smooth curves, called \emph{characteristic curves}. In this framework, the function $u$ can typically be written in a compositional form of two functions where one is  high-dimensional and smooth and the other is low-dimensional and (potentially) rough. 

Based on this split and the inherent compositionality of neural networks, we will demonstrate that every $u$ satisfying \eqref{eq:Motivation} can be approximated efficiently by neural networks with ReLU activation function. The approximation rate is significantly better than that of any classical regularity-based approximation of $u$. In particular, in the prescribed set-up, we will observe an \emph{approximation rate independent of the dimension $D$ of the parameter space}. 

The material presented below was first established in a mini-project of the first author at the University of Oxford, \cite{FabianTT}.

\subsection{Applications and relevance}

We believe that the efficient approximation of solutions of parametric transport equations with dimension-independent approximation rates is an interesting and relevant problem in the following domains:

\begin{itemize}
    \item \emph{Approximation theory:}
   The class of functions that are solutions of high-dimensional parametric linear transport equations is a relevant but non-standard function class. This class while high-dimensional has a non-trivial but rigid structure imposed upon via the underlying PDE. It is, therefore, interesting to establish to what extent deep neural networks can leverage on this structure. In this context, similar approximation schemes based on structured systems were developed for special types of (parametric) transport equations. For example, in \cite{dahmen2014efficient, dahmen2018adaptive} and \cite{obermeier2019approximation} systems were introduced that approximate the solutions of linear transport equations with linear or $C^2$-regular characteristic curves. These constructions are closely tied to the type of characteristic curves. We demonstrate here that, in contrast to these systems, approximation by deep neural networks is automatically adapted to the underlying regularity of the problems and benefits from higher regularity of the characteristic curves.

    \item \emph{Estimation:} In machine learning and especially in deep learning, deep neural networks are trained with gradient-based optimisation algorithms to minimise empirical energies based on random samples, \cite{Goodfellow-et-al-2016, LeCun2015DeepLearning}. These techniques have proven to be extremely successful in a variety of applications. 
    
    Consider a parametric transport problem of the form \eqref{eq:Motivation} where $u_0$, $f$, and $V$ are unknown, but samples $(u(t_i,x_i, \eta_i))_{i=1}^N$ are available through measurements. Such a scenario could be encountered in the transport of pollution in fluids under unknown circumstances, but with a control on parameters of the experiment. In this formulation, the transport problem is a standard supervised learning problem. Moreover, classical methods to solve linear transport equations cannot be used at all without knowledge of $f$ and $V$ in \eqref{eq:Motivation}.
    
    Certainly, this estimation problem can only be successfully solved with deep learning techniques, if the correct solution to the problem can be represented or closely approximated by a deep neural network. In this context, our results show the feasibility of this approach.      
    
    \item \emph{Numerical analysis:} 
    Deep neural networks have been employed as Ansatz spaces for PDEs in multiple settings before \cite{sirignano2018dgm, weinan2018deep, beck2018solving}. Of course, the efficiency of these methods depends on the capacity of the Ansatz space to capture the true solution. 
    
    Our proposed approximation of solutions of \eqref{eq:Motivation} by deep neural networks can be thought of as a higher-order method that automatically adapts to the regularity of the underlying characteristic curves. 
    
    An established method to solve \eqref{eq:Motivation} is by using (Petrov-) Galerkin-type discretisations, \cite{dahmen2020adaptive, egger2012mixed}. These methods are, however, typically not adaptive to singularities lying on lower-dimensional manifolds. In the model of \eqref{eq:Motivation}, such structured singularities evolve precisely along the characteristic curves. While some advances to handle structured singularities have been made, e.g. \cite{dahmen2014efficient}, the adaptivity to the manifold only uses low-order information on the smoothness of the manifold. Our results demonstrate that an approximation via deep neural networks adapts to any regularity of the characteristic curves in the sense that the approximation quality improves for smoother characteristic curves.
    
    Standard discretisation and time-stepping techniques such as finite differences and Euler, Crank-Nicolson, or higher-order variants converge with rates depending on the global regularity of the solution of the PDE which, in our set-up, is assumed to be quite low.

    \item \emph{Reduced-Order Models:} In applications where the solution of \eqref{eq:Motivation} is requested for many different parameter values, it is desirable to employ model reduction techniques \cite{hesthaven2016certified, quarteroni2014reduced}. It is well known that parametric linear transport problems are highly challenging for linear reduced-order models because the dimension of linear approximation spaces to capture the non-linear evolution of singularities can be excessive, \cite[Section 5]{ohlberger2015reduced}, \cite[Section 6.3]{dahmen2014double}. Indeed, linear reduced-order models typically succeed only if additional assumptions are made on the parameter dependence, such as a certain separability of the parametric dependence and the spatial dependence of $V$, \cite{grella2013sparse}.

    The neural network based approximation presented in this work requires almost no structural assumption on the parametric dependence. Indeed, a smooth dependence of $V$ and $f$ on the parameters is sufficient. 
    In this context, a superiority of deep neural network-based approaches over linear reduced-order models to solve parametric transport equations was also empirically observed in \cite{fresca2020comprehensive}.
\end{itemize}

\subsection{Related work}

This work describes the capacity of neural networks to approximate high-dimensional functions with asymptotic rates independent of the underlying dimension. Of course, approximation theory of deep neural networks is a well established field. Therefore, in order to place our contribution in the context of existing literature, we provide an overview of classical and more modern developments in the field below.

\subsubsection{Classical approximation}
The first and probably most prominent result describing the approximation capabilities of neural networks is the universal approximation theorem, \cite{cybenko1989approximation, hornik1989multilayer}. This theorem states that, on a compact, domain, every continuous function can be arbitrarily well approximated by a neural network in the uniform norm. These statements, however, do not provide an estimate on the required sizes of the approximating neural networks. The typical approach to obtain a quantitative estimate on the order of approximation is to re-approximate classical methods. For example, in \cite{mhaskar1993approximation} and \cite{MhaskarAnalytic1996}, it was shown that neural networks yield the same approximation rates as splines when approximating smooth functions.

Recently, approximation by neural networks with the ReLU activation function has received the most attention since this activation function is arguably the most widely-used in applications. It was demonstrated that deep neural networks with the ReLU activation function achieve the same approximation rates as linear and higher-order finite elements \cite{LinearFEMReLU, opschoor2019deep}, wavelets \cite{shaham2018provable}, and local approximation by Taylor polynomials, \cite{yarotsky2017error}.

These classical approximation results show that deep neural networks are very versatile by combining the approximation capabilities of a wide variety of classical tools. However, they do not identify a particular situation where deep neural networks outperform the best classical method. This picture changes drastically, when one considers high-dimensional approximation.

\subsubsection{High-dimensional approximation}

High-dimensional approximation generally suffers from a curse of dimension, meaning that approximation rates deteriorate exponentially with increasing dimension, \cite{bellman1952theory, novak2009approximation}. Nonetheless, if an additional structure is assumed, then the curse of dimension can be overcome. It turns out that deep neural networks can take advantage of a wide variety of complex additional structural properties. For example, it was shown in \cite{barron1993universal}, that deep neural networks can approximate functions with bounded first Fourier moments without a curse of dimension. Other regularity-based assumptions were used in \cite{montanelli2019deep} and \cite{suzuki2018adaptivity}. Further classes of functions with structural assumptions such as functions based on directed acyclic graphs \cite{poggio2017and} or functions admitting strong invariances \cite[Section 5]{petersen2018optimal} allow similar results. Finally, if the approximation error is evaluated on a low-dimensional manifold only, then \cite{shaham2018provable, schmidt2019deep, chen2019efficient, boelcskei2019optimal} show approximation rates independent of the ambient dimension.    

\subsubsection{Approximation of solutions of PDEs}

The extraordinary efficiency in the approximation of certain high-dimensional functions has been especially interesting in connection with the numerical solution of PDEs \cite{sirignano2018dgm, weinan2018deep, beck2018solving}. For example, for high-dimensional Black Scholes-, Kolmogorov-, or heat equations deep neural networks can efficiently approximate the solutions thereof in a regime where any mesh-based method would fail, \cite{elbrachter2018dnn, hutzenthaler2019proof, beck2019full, berner2018analysis}. In these works, a compositional structure of the solution of a PDE is derived via the Feynman-Kac formula. The approach via the method of characteristics of our work can be interpreted as a special case of the approach via the Feynman-Kac formula. 

Moreover, in the framework of parametric problems, high-dimensional problems can be efficiently represented if there exist suitable representations thereof in a general reduced basis, \cite{kutyniok2019theoretical}, or as a polynomial chaos expansion, \cite{SchwabZechHighD2019, OSZ19_839}. 

\subsection{Outline}
In Section \ref{sec:neuralNetworks}, we introduce all notions and fundamental results associated with neural networks. Section \ref{sec:linTransEq} is devoted to the introduction of various types of linear transport equations. In Section \ref{chap:DNNApprox}, we present the four main results of this work: Theorems \ref{thm:mainresult}, \ref{thm:mainWeak}, \ref{thm:resultsourceterm}, and \ref{thm:resultcons}.
These results describe approximation rate bounds for the solutions of the equations of Section \ref{sec:linTransEq} by deep neural networks. Finally, in Section \ref{sec:Extensions}, we discuss natural extensions of the presented results. Some auxiliary results have been deferred to the appendix.

\subsection{Notation}
Below we collect some notation that is used throughout the manuscript. This notation is mostly standard and hence this section can be skipped and only be referred to when a symbol is unclear. 

\medskip

 We denote by $\N=\{1,2,...\}$ the set of all \emph{natural numbers} and define, for $k \in \N$, the set $\N_{\geq k} \coloneqq \{n \in \N \colon n \geq k\}$. For $d_1,d_2\in\N$ we denote by $\mathrm{Id}_{\R^{d_1}}$ the \emph{identity} on $\R^{d_1}$ and by ${0}_{\R^{d_1 \times d_2}}$ we denote the map from $\R^{d_1}$ to $\R^{d_2}$ that vanishes everywhere. We denote by ${0}_{\R^{d_1}}$ the \emph{zero vector} in $\R^{d_1}$. On $\R^{d_1 \times d_2}$ we denote by $\|\cdot\|$ the \emph{euclidean norm} and by $\|\cdot\|_\infty$ the \emph{maximum norm}. The \emph{number of nonzero entries} of a matrix or vector $A\in\R^{d_1 \times d_2}$ is counted by $\|\cdot\|_{0}$, where $		\|A\|_{0} \coloneqq |\{(i,j): A_{i,j}\neq 0\}|$. 
 
If $d_1,d_2,d_3\in\N$, and $A\in\R^{d_1\times d_2}, B\in\R^{d_1 \times d_3}$, then we use the \emph{block matrix notation} and write for the horizontal concatenation of $A$ and $B$ 
\[
\bmat{c c}{A & B} \in\R^{d_1,d_2+d_3}\quad\text{or}\quad \bmat{c|c}{A & B}\in\R^{d_1,d_2+d_3},
\]
where the second notation is used if a stronger delineation between different blocks is appropriate. A similar notation is used for the vertical concatenation of $A\in\R^{d_1\times d_2}$ and $B\in\R^{d_3 \times d_2}$. 

For $d_1, d_2 \in \N$, and $\Omega \subset \R^{d_1}$, we denote by $L^p(\Omega,\R^{d_2}), p\in [1,\infty]$ the \emph{$\R^{d_2}$-valued Lebesgue spaces}, where we set $L^p(\Omega)\coloneqq L^p(\Omega,\R)$. For $k\in \N$, we denote by 
$W^{k, \infty}(\Omega)$, the space of \emph{$k$-times weakly differentiable functions that have all derivatives of order at most $k$ in $L^\infty(\Omega)$}. 
The space $W^{k, \infty}_{\loc}(\Omega)$ consists of functions such their restriction to every compact $K \subset \Omega$ is in $W^{k, \infty}(K)$. 
By $C^k(\Omega,\R^{d_2})$, we denote the set of \emph{$k$-times continuously differentiable functions} mapping from $\Omega$ to $\R^{d_2}$, where we set $C^k(\Omega)\coloneqq C^k(\Omega,\R)$. By $C^k_c(\Omega)$ we denote all functions in $C^k(\Omega)$ that have compact support.

For a Lipschitz continuous function $f: \R^{d_1} \mapsto \R^{d_2}$ we denote 
$$
\mathrm{Lip}_{f} \coloneqq \sup_{x \neq y} \frac{\|f(x) - f(y)\|}{ \|x-y\|}.
$$

Let $a>0$, then we say for two functions $f:(0,a)\to [0,\infty)$ and $g:(0,a)\to[0,\infty)$ that $f(\eps)$ is in $\mathcal{O}(g(\eps))$ for $\eps \to 0$ if there exists $0<\delta<a$ and $C>0$ such that $f(\eps)\leq C g(\eps)$ for all $\eps\in (0,\delta)$.

\section{Neural Networks}\label{sec:neuralNetworks}

In this section, we define neural networks and then recall a couple of operations on these objects that will be used frequently in the sequel. In the definition of neural networks, we distinguish between a neural network as a set of weights and an associated function that we call the realisation of the neural network. This formal approach was introduced in \cite{petersen2018optimal}, but we recall here a slightly different formulation of \cite{guhring2019error} for neural networks that allow so-called skip connections.

\begin{definition}
	Let $d,L\in\N$. A \emph{neural network (NN) $\Phi$ with input dimension $d$ and $L$ layers} is a sequence of matrix-vector tuples
	\[
		\Phi=((A_1,b_1),(A_2,b_2),\dots,(A_L,b_L)),
		\]
		where, for $N_0=d$ and $N_1,\ldots,N_L\in\N$, each $A_\ell$ is an $N_\ell\times \sum_{k=0}^{\ell-1} N_k$ matrix, and $b_\ell\in\R^{N_\ell}$.

    Let $\varrho:\R\to\R$ be the \emph{ReLU}, i.e., $\varrho(x) = \max\{0,x\}$ and let $\Phi$  be a NN as above. Then we define the associated \emph{realisation of $\Phi$} as the map $\Realization(\Phi):\R^d\to\R^{N_L}$ such that 
		\[
			\Realization(\Phi)(x)=x_L,
			\]
			where $x_L$ results from the following scheme:
			\begin{align*}
				&x_0:=x,\\
				&x_\ell:=\varrho\left(A_\ell\bmat{c|c|c}{x_0^T & \ldots &  x_{\ell-1}^T}^T+b_l\right), \quad\text{for }\ell=1,\ldots L-1,\\
				&x_L:=A_L \bmat{c|c|c}{x_0^T & \ldots &  x_{L-1}^T}^T+b_L.
			\end{align*}
			Here $\varrho$ acts componentwise, i.e., $\varrho(y)=[\varrho(y^1),\ldots,\varrho(y^m)]$ for $y=[y^1,\ldots,y^m]\in\R^m$. We sometimes write $A_\ell$ in block-matrix form as
			\[
				A_\ell=\bmat{c|c|c}{
					A_{\ell,0} & \ldots & A_{\ell,\ell-1}
					},
				\]
				where $A_{\ell,k}$ is an $N_\ell \times N_k$ matrix for $k=0,\ldots,\ell-1$ and $\ell=1,\ldots,L$. Then
			\begin{align*}
				&x_\ell = \varrho\left(A_{\ell,0} x_0+\ldots+ A_{\ell,\ell-1} x_{\ell-1}+b_\ell\right), \quad\text{for }\ell=1,\ldots L-1,\\
				&x_L=A_{L,0} x_0+\ldots+ A_{L,L-1} x_{L-1}+b_L.
			\end{align*}

We call $N(\Phi):=d+\sum_{j=1}^L N_j$ the \emph{number of neurons} of the NN $\Phi$, $L=L(\Phi)$ the \emph{number of layers}, and $W(\Phi) \coloneqq \sum_{j=1}^L (\|A_j\|_{0}+\|b_j\|_{0})$ is called the \emph{number of weights} of $\Phi$. Moreover, we refer to $N_L$ as \emph{output dimension} of $\Phi$.
\end{definition}

\subsection{Standard operations on neural networks}

We collect four standard operations that can be performed with NNs below. First, we can concatenate two NNs $\Phi^1,\Phi^2$ in such a way that the realisation of the concatenation is a composition of the individual realisations of $\Phi^1,\Phi^2$.

\begin{proposition}[{\cite[Remark 2.8]{guhring2019error}}] \label{prop:SparseConc}
	Let	$\Phi^1, \Phi^2$ be two NNs such that the input dimension $d$ of $\Phi^1$ is equal to the output dimension of $\Phi^2$. Then there exists a NN $\Phi^1\odot\Phi^2$ such that 
	\begin{itemize}
	    \item $L\left(\Phi^1\odot\Phi^2\right) = L\left(\Phi^1\right) + L\left(\Phi^2\right)$,
	    \item $W\left(\Phi^1\odot\Phi^2\right) \leq 2 W\left(\Phi^1\right) + 2 W\left(\Phi^2\right)$,
	    \item $\Realization\left(\Phi^1\odot\Phi^2\right)(x) = \Realization\left(\Phi^1\right) \circ \Realization\left(\Phi^2\right)(x)$ for all $x \in \R^d$.
	\end{itemize}
	We call $\Phi^1\odot\Phi^2$ the \emph{sparse concatenation of $\Phi^1$ and $\Phi^2$.}
\end{proposition}

An additional operation that is frequently applied to NNs in the sequel is that of parallelisation. This procedure puts NNs in parallel such that the output of the realisation is a vector containing the outputs of the original NNs. 

\begin{proposition}[{\cite[Remark 2.10]{guhring2019error}}]\label{prop:parallelization}
	Let $n, d\in\N$ and, for $i=1,\ldots,n$, let $\Phi^i$ be a NN with $d$-dimensional input and $L_i \in \N$ layers.	Then there exists a NN $\mathrm{P}(\Phi^1,\dots, \Phi^n)$ 
	with $d$-dimensional input such that 
	\begin{itemize}
	    \item $L\left(\mathrm{P}\left(\Phi^1,\dots, \Phi^n\right)\right) = \max\{L_1,\ldots,L_n\}$, 
	    \item $W\left(\mathrm{P}\left(\Phi^1,\ldots,\Phi^n\right)\right)=\sum_{i=1}^n W\left(\Phi^i\right)$,
	    \item $\Realization\left(\mathrm{P}\left(\Phi^1,\ldots,\Phi^n\right)\right)(x)=\left(\Realization\left(\Phi^1\right)(x),\ldots,\Realization\left(\Phi^n\right)(x)\right)$ for all $x \in \R^d$.
	\end{itemize}
    We call $\mathrm{P}(\Phi^1,\dots, \Phi^n)$ the \emph{parallelisation of $\Phi^1, \dots, \Phi^n$.}	
\end{proposition}

We will occasionally need to construct NNs the realisation of which is the sum of functions that we had approximated beforehand by realisations of NNs. In this situation, the following operation that emulates a sum of NNs is convenient.

\begin{proposition}[Sums of NNs]\label{def:addofNN}
Let $d\in\N$, and $\Phi^1, \Phi^2$ be two NNs with $d$-dimensional input and one-dimensional output.	Then there exists a NN $\Phi^1 \oplus \Phi^2$ with $d$-dimensional input such that 
	\begin{itemize}
	    \item $L\left(\Phi^1 \oplus \Phi^2\right) = \max\{L(\Phi^1), L(\Phi^2)\}$, 
	    \item $W\left(\Phi^1 \oplus \Phi^2\right) = W(\Phi^1) + W(\Phi^2)$,
	    \item $\Realization\left(\Phi^1 \oplus \Phi^2\right)(x)= \Realization(\Phi^1)(x) + \Realization(\Phi^2)(x)$ for all $x \in \R^d$.
	\end{itemize}
    We call $\Phi^1 \oplus \Phi^2$ the \emph{sum of $\Phi^1$ and $\Phi^2$.}	
\end{proposition}
\begin{proof}
Let 
\begin{align*}
((A_1,b_1),(A_2,b_2),\dots,(A_L,b_L)) \coloneqq \mathrm{P}(\Phi^1, \Phi^2).
\end{align*}
Then we set 
$$
\widetilde{A}_L \coloneqq \bmat{cc}{1 & 1} A_L \quad \text{ and } \widetilde{b}_L \coloneqq \bmat{cc}{1 & 1} b_L. 
$$
Clearly, $\|\widetilde{A}_L\|_{\ell_0} \leq  \|{A}_L\|_{\ell_0}$ and $\|\tilde{b}_L\|_{\ell_0} \leq  \|b_L\|_{\ell_0}$. We define 
$$
\Phi^1 \oplus \Phi^2 \coloneqq \left((A_1,b_1),(A_2,b_2),\dots,(A_{L-1},b_{L-1}), \left(\widetilde{A}_L,\tilde{b}_L\right)\right).
$$
Per construction, 
$$
\Realization\left(\Phi^1 \oplus \Phi^2\right)(x) = \bmat{cc}{1 & 1} \left(\begin{array}{c}
    \Realization\left(\Phi^1\right)(x)\\ \Realization(\Phi^2)(x)
\end{array}\right) = \Realization\left(\Phi^1\right)(x) + \Realization\left(\Phi^2\right)(x) \quad \text{  for every }x \in \R^d. 
$$
\end{proof}

Finally, we can construct a NN that represents the multiplication of two NNs $\Phi^1$ and $\Phi^2$ in the sense that its realisation is close to the multiplication of the realisations of $\Phi^1$ and $\Phi^2$. In contrast to the previous operations, this emulation of the multiplication is not exact but requires a parameter $\eps > 0$ describing how accurately the multiplication is implemented.

\begin{proposition}[Multiplication of NN] \label{prop:multiplication}
Let  $\Phi^1, \Phi^2$ be NNs with input dimensions $d_1$ and $d_2$ and output dimension $1$. Then, for every $\eps \in (0,1)$, there exists a NN $\Phi^1 \otimes^\eps \Phi^2$ such that, for a universal constant $c_1 >0$ and for $c_2= c_2(\|\Realization(\Phi^1)\|_{L^\infty},\|\Realization(\Phi^2)\|_{L^\infty}) > 0$, there holds 
\begin{itemize}
    \item $L\left(\Phi^1 \otimes^\eps \Phi^2\right)  \leq \max\{L\left(\Phi^1\right),L\left(\Phi^2\right)\} + c_1 \ln(1/\eps) + c_2$,
    \item $W\left(\Phi^1 \otimes^\eps \Phi^2\right) \leq c_1 \ln(1/\eps) + c_2+  2W\left(\Phi^1\right)  + 2W(\Phi^2)$,
    \item $\left\|\mathrm{R}\left(\Phi^1 \otimes^\eps \Phi^2\right) -\mathrm{R}\left(\Phi^1\right) \mathrm{R}\left(\Phi^2\right)\right\|_{L^\infty} \leq \eps$.
\end{itemize}
\end{proposition}
\begin{proof}
By \cite[Proposition 3]{yarotsky2017error}, there exists, for every $\eps\in (0,1)$ and $M \in \N$, a NN $\times^{\eps,M}$ with two-dimensional input and one-dimensional output satisfying 
$$
\left|\Realization\left(\times^{\eps,M}\right)(x,y) - xy\right| \leq \eps,
$$
for all $x,y \in [-M,M]$. Moreover, 
$$
W\left(\times^{\eps,M}\right) \leq c_2 - c_1 \ln(\eps),
$$
for a universal constant $c_1$ and $c_2 = c_2(M)$.

We set $\widetilde{M} \coloneqq \max \{\|\Realization(\Phi^1)\|_{L^\infty},\|\Realization(\Phi^2)\|_{L^\infty} \}$ and define
$$
\Phi^1 \otimes^\eps \Phi^2 \coloneqq \times^{\eps,\widetilde{M}}
\sconc \mathrm{P}(\Phi_1, \Phi_2).
$$
The result now follows from Propositions \ref{prop:SparseConc} and \ref{prop:parallelization}.
\end{proof}

\subsection{Approximation of smooth functions}

In addition to the operations on NNs described in the previous section, we will frequently invoke the following standard approximation result of smooth functions by realisations of NNs.

\begin{theorem}[{\cite[Theorem 1]{yarotsky2017error}}]\label{thm:smoothFunctionsApprox}
Let $k,d \in \N$ and 
\begin{align*}
    F_{k,d}\coloneqq \left\{f \in W^{k,\infty}\left([0,1]^d\right) \, \colon \, \|f\|_{W^{k,\infty}([0,1]^d)} \leq 1\right\}.
\end{align*}
Then there exists $c=c(k,d) > 0$ such that, for every $f \in F_{k,d}$ and every $\eps \in (0,1)$, there exists a NN $\Phi^{f,\eps}$ with $d$-dimensional input such that, 
\begin{itemize}
\item $L(\Phi^{f,\eps})\leq c\cdot (\ln(1/\eps)+1)$,
\item $W(\Phi^{f,\eps})\leq c\,\eps^{-d/k}\cdot (\ln(1/\eps)+1)$,
\item $\|f-\Realization(\Phi^{f,\eps})\|_{L^\infty} < \eps$.
\end{itemize}
\end{theorem}

\begin{remark}\label{rem:smoothFunctionsApprox}
\begin{itemize}
    \item[(i)] The norm we use for $W^{k,\infty}([0,1]^d)$ is
    \begin{align*}
        \|f\|_{W^{k,\infty}([0,1]^d)} \coloneqq \max_{\alpha: |\alpha|\leq k} \esssup_{x \in [0,1]^d} |D^\alpha f(x)|.
    \end{align*}
    \item[(ii)] The space $W^{k, \infty}([0,1]^d)$ can be identified with the set of $k-1$-times continuously differentiable functions all derivatives of order $k-1$ of which are Lipschitz continuous.
    \item[(iii)] If we consider the ball with radius $R$ in $W^{k, \infty}([0,1]^d)$, i.e.
    \begin{align*}
        F^R_{k,d} \coloneqq \Big\{f \in W^{k,\infty}([0,1]^d) \, \colon \, \|f\|_{W^{k,\infty}([0,1]^d)} \leq R\Big\}
    \end{align*}
    instead of the unit ball $F_{k,d}$ then the constant $c$ from Theorem~\ref{thm:smoothFunctionsApprox} also depends on $R$. However, the asymptotic behaviour with respect to $\eps$ remains unchanged. 
    The same change holds if we consider the space $W^{k,\infty}([0,R]^d)$ instead of $W^{k, \infty}([0,1]^d)$.
\end{itemize}
\end{remark}

\section{Linear Transport Equations}\label{sec:linTransEq}

In this section, we introduce the Cauchy problem for the parametric linear transport equation, state the most important existence results for several types of linear transport equations and provide expressions for their solutions. Here, we mainly follow \cite{GolseFrench}. An English translation of this source can be found in the lecture notes \cite{LectureNotesMeanField}. For more information on linear transport equations see also \cite{Ambrosio2008,Ambrosio2005, Golse2000}.
\begin{definition}
The Cauchy problem of the parametric linear transport equation is given by 
\begin{numcases}{\label{eq:transportequation}}
\partial_t u(t,x,\eta) + V(t,x,\eta)\cdot \nabla_x u(t,x,\eta)=0,\\
u(0,x,\eta)=u_0(x),
\end{numcases}
where $t\in [0,T]$, $x\in \R^n$, and $\eta \in [0,1]^D$ for some $n, D \in \N$, and $T>0$. The vector field $V \in C^k([0,T]\cross \R^n \cross [0,1]^D; \R^n)$ and the initial condition $u_0 \in C^s(\R^n; \R)$ are given with $s, k \in \N$.
\end{definition}

It is well known that linear transport equations can be solved via the method of characteristics \cite{Courant1989, Evans2010,John1978}. 
The idea of this method is to consider characteristic curves that are defined so that the solution $u$ of \eqref{eq:transportequation} is constant along these curves. Then the solution at a point $(t,x,\eta)$ equals the initial data evaluated at the origin of this curve. 

\begin{definition}
The \emph{characteristic curve} of the transport operator $\partial_t + V(t, x, \eta) \cdot \nabla_x$ passing through $x$ at time $s=t$ is given by the set $\{(s,\gamma(s))\ \colon \ s \in [0,T]\}$, where $\gamma$ is the solution of the \emph{characteristic system of ordinary differential equations} 
\begin{numcases}{\label{gammaprime}}
\dot{\gamma}(s)=V(s,\gamma(s),\eta),\\
\gamma(t)=x. 
\end{numcases}
\end{definition}

Let us briefly show why the solution $u$ of \eqref{eq:transportequation} does not change along characteristic curves. Considering the case where $V(t,x,\eta)\equiv v$ for a $v\in \R^n$ and dropping the $\eta$-dependency, we have
\begin{align*}
    \frac{\mathrm{d}}{\dt} u(t,\gamma(t))&= \del_t u(t,\gamma(t)) + \grad_x u(t,\gamma(t))\cdot \dot{\gamma}(t) \nonumber\\
    &= \del_t u(t,\gamma(t)) + \grad_x u(t,\gamma(t))\cdot v\\
    &= (\del_t + V(t,x)\cdot\grad_x)u(t,\gamma(t))=0,\nonumber
\end{align*}
where the last equality is due to \eqref{eq:transportequation}.

To make the method of characteristics work, we have to ensure that the characteristic curves are diffeomorphisms and that there exists a global solution of system $\eqref{gammaprime}$. Therefore, we make the following assumptions on the vector field $V$: Let $n,D \in \N$, and $T>0$.
\begin{itemize}
    \item[\quad (H1)] For some $k \in \N$ there holds $V \in C^k([0,T]\times \R^n\times[0,1]^D; \R^n)$. 
    \item[\quad (H2)] There exists a $C > 0$ s.t.
\begin{align*}
|V(t,x,\eta)| \leq C\, (1+|x|) \quad \text{ for all } (t,x,\eta) \in [0,T]\times\R^n\times [0,1]^D.
\end{align*}
\end{itemize}

The following theorem states that these assumptions lead to global existence of the characteristic curves and characterises their regularity.
\begin{theorem}[{\cite[Theorem 2.2.2]{LectureNotesMeanField}}]\label{thm:ExistenceX}
Let $V$ satisfy (H1) and (H2) with $n,D,k\in \N$, and $T>0$. Then, for all $(t,x,\eta) \in [0,T]\cross \R^n \cross [0,1]^D$, the system of \eqref{gammaprime} has a unique solution $\gamma \in C^{k}([0,T])$.
Furthermore, the map X defined by 
\begin{align*}
       X(s,t,x,\eta) \coloneqq \gamma(s)
\end{align*}
is in $C^{k}([0,T]\cross [0,T]\cross \R^n \cross [0,1]^D)$.
\end{theorem}

\begin{proof}
The proof presented in \cite{LectureNotesMeanField} can directly be extended to the parametric case. Moreover, the differentiability with respect to $\eta$ is a standard result, compare \cite[Corollary 4.1]{Hartman2002}.
\end{proof}

For our main result, we want to approximate the map $X$ with a NN. To quantify the complexity of this NN, we need a bound on the $C^k$-norm of $X$ in terms of the given data. We establish this bound in the next proposition.

\begin{proposition}\label{prop:boundX}
 Let $V$ satisfy the conditions (H1) and (H2) with $n,D,k\in \N$ and $T>0$. Then for every $K\subset \R^n$ compact there exists a constant $G = G(k,d,T,|K|,\|V\|_{C^k} ) > 0$ such that 
\begin{align*}
    \|X\|_{C^{k}([0,T]\cross [0,T]\cross K \cross [0,1]^D)} \leq G. 
\end{align*}
\end{proposition}
\begin{proof}
The proof is presented in Appendix \ref{app:boundX}.
\end{proof}

\subsection{Solutions of the standard linear transport equation}

The next theorem states the existence of a solution for the linear transport equation of \eqref{eq:transportequation}. Furthermore, the theorem establishes that the solution has a compositional structure resulting from composing the initial data with the solution of the characteristic system of ODEs starting at $(t,x,\eta)$ evaluated at $s=0$.
\begin{theorem}[{\cite[Theorem 2.2.4]{LectureNotesMeanField}}]\label{thm:solutionTE}
Let $V$ satisfy the assumptions (H1) and (H2) with $n,D,k\in \N$, and $T>0$. Further, let $u_0 \in C^s(\R^n)$, $s \in \N$. Then the Cauchy problem for the parametric linear transport equation of \eqref{eq:transportequation}
has a unique solution $u \in C^{\min\{s,k\}}([0,T]\times\R^n\cross [0,1]^D)$ which is given by 
\begin{align*}
    u(t,x,\eta)=u_0(X(0,t,x,\eta)).
\end{align*}
\end{theorem}

For initial conditions that are not differentiable it makes sense to introduce a weak notion of a solution. The following definition and proposition where taken from {\cite{Lions1989}}. A proof for the simplified case, where $\mathrm{div}_x V=0$ can be found in \cite[Theorem 3.12]{LectureNotesTheoryTransport}.
\begin{definition}[{\cite{Lions1989}}]\label{def:weaksol}
Let $V$ satisfy the assumptions (H1) and (H2) with $n,D,k\in \N$, and $T>0$. 
Further, let $u_0 \in L^\infty(\R^n)$. A weak solution to \eqref{eq:transportequation} is a function $u \in L^\infty([0,T]\times\R^n\times[0,1]^D)$ which satisfies the \emph{weak formulation}
\begin{align}
    &\int_0^T \int_{\R^n} u(t,x,\eta) \left[\del_t\varphi + V(t, x, \eta)\cdot \grad_x \varphi(t, x) + \mathrm{div}_x V (t, x, \eta) \varphi(t, x)\right] \dx \dt \nonumber\\
    &\qquad + \int_{\R^n} u_0(x)\varphi(0,x) \dx =0\label{eq:weakForm}
\end{align}
for all $\varphi \in C^1_c([0,T)\times\R^n)$ and all $\eta \in [0,1]^D.$
\end{definition}
As for strong solutions of the transport equation, the solution of the weak formulation is given by a composition of the initial condition with a flow along the characteristic curves. 

\begin{proposition}[{\cite{Lions1989}}]\label{prop:weak}
Let $V$ satisfy the assumptions (H1) and (H2) with $n,D,k\in \N$, and $T>0$. 
Further, let $u_0 \in L^\infty(\R^n)$. Then there exists a global weak solution $u \in  L^\infty([0,T]\times\R^n\times[0,1]^D)$ to \eqref{eq:transportequation} which is given by
\begin{align*}
    u(t,x,\eta)=u_0(X(0,t,x,\eta)).
\end{align*}
\end{proposition}

\subsection{Solutions of extensions of the parametric linear transport equation}

In Chapter \ref{chap:DNNApprox}, we extend our main result to linear transport equations that include source terms and are formulated in conservative form. The following two propositions present the corresponding existence results and the form of the solutions for problems with source terms and in conservative form.
\begin{proposition}[{\cite[Theorem 3.9]{LectureNotesTheoryTransport}}]\label{prop:source}
Let $V$ satisfy assumptions (H1) and (H2) with $n,D,k\in \N$, and $T>0$. Further, let $u_0 \in C^s(\R^n)$ and $f\in C^{s'}([0,T]\times\R^n\times[0,1]^D)$, where $s, s' \in \N$. Then the Cauchy problem for the non-homogeneous parametric linear transport equation
\begin{numcases}{\label{eq:sourcetermeq}}
\del_t u(t,x,\eta) + V(t,x,\eta)\cdot \grad_x u(t,x,\eta)=f(t,x,\eta),\\
u(0,x,\eta)=u_0(x),
\end{numcases}
has a unique solution $u \in C^{\min\{s,s',k\}}([0,T]\times\R^n\cross [0,1]^D)$ which is given by
\begin{align}\label{eq:formsolsource}
    u(t,x,\eta)=u_0(X(0,t,x,\eta)) + \int_{0}^t f(s,X(s,t,x,\eta),\eta) \ds.
\end{align}
\end{proposition}
\begin{remark}\label{rem:WeakFormulationSource}
Similar to Proposition \ref{prop:weak}, one can prove the existence and uniqueness of a weak solution with a source term $f\in C^0([0,T]\times \R^n\times [0,1]^D)$. 
In this case, the associated weak formulation is given by \eqref{eq:weakForm} after replacing the right-hand side by $\int_{0}^T \int_{\R^n} f(t,x,\eta)\varphi(t,x,\eta) \dx \dt$. The weak solution $u \in L^\infty([0,T]\times\R^n\times[0,1]^D)$ of that problem is still given by \eqref{eq:formsolsource}. Here one only needs to assume that $u_0$ is continuous, \cite[Remark 3.13]{LectureNotesTheoryTransport} or \cite{Lions1989}.
\end{remark}
\begin{proposition}[{\cite[Theorem 2.3.6]{LectureNotesMeanField}}]\label{prop:cons}
Let $V$ satisfy assumptions (H1) and (H2) with $n,D,k\in \N$, and $T>0$. Further, let $u_0 \in C^s(\R^n)$, $s \in \N$. Then the Cauchy problem for the conservative parametric linear transport equation
\begin{numcases}{\label{eq:conservative}}
\del_t u(t,x,\eta) +\mathrm{div}_x(V(t,x,\eta) u(t,x,\eta))= 0 , \\
u(0,x,\eta)=u_0(x),
\end{numcases}
has a unique solution $u \in C^{\min\{s,k\}}([0,T]\times\R^n\cross [0,1]^D)$ which is given by
\begin{align}\label{eq:formsolcons}
    u(t,x,\eta)=u_0(X(0,t,x,\eta))J(0,t,x,\eta)
\end{align}
with 
\begin{align*}
    J(s,t,x,\eta) = \det(D_x X(s,t,x,\eta)).
\end{align*}
\end{proposition}
\begin{remark}\label{rem:WeakSolutionConservative}
Again, one can show that there exists a unique weak solution for the conservative formulation which is given by \eqref{eq:formsolcons}. See \cite[Section 2.3]{LectureNotesMeanField} for more information about this problem.
\end{remark}
\begin{remark}
The conservative form \eqref{eq:conservative} simplifies to the original linear transport equation \eqref{eq:transportequation} if $\mathrm{div}_x V=0$.
\end{remark}

\section{DNN Approximation of Solutions of Linear Transport Equations}\label{chap:DNNApprox}

Theorem~\ref{thm:solutionTE} and Propositions \ref{prop:source} and \ref{prop:cons} suggest that the solutions of parametric linear transport equations are of a compositional form, where the initial condition is composed with a flow along characteristic curves. 
Since realisations of NNs are naturally of compositional structure, it is therefore conceivable that the form of the solutions of linear transport equations can be efficiently resolved by NNs. 
Indeed, based on this observation we present, for each of the cases discussed in Section \ref{sec:linTransEq}, an approximation result for the solution of the associated parametric linear transport equations by NNs.

\subsection{Standard linear transport equations}

We start by presenting an approximation result for the solutions of standard linear transport equations as described in Theorem~\ref{thm:solutionTE}. We will assume that the initial condition can be approximated reasonably well by NNs. For this, we use the following definition:

\begin{definition}\label{def:approximability}
Let $n \in \N$ and $r >0$, a function $f \in L^\infty(\R^n)$ is \emph{$r$-approximable by NNs} if, for every compact set $K \subset \R^n$, there exists a constant $c = c(K,r,f) > 0$ such that for every $\epsilon \in (0,1)$ there exists a NN $\Phi^{f, \eps}$ such that 
\begin{itemize}
    \item $L\left(\Phi^{f, \eps}\right) \leq c\cdot \left(\ln(1/\eps) + 1\right)$,
    \item $W\left(\Phi^{f, \eps}\right) \leq c \, \eps^{-1/r}$,
    \item $\left\|f - \Realization(\Phi)\right\|_{L^\infty(K)} \leq \eps$.
\end{itemize}
\end{definition}
\begin{remark}\label{rem:RApproximable}
By Theorem \ref{thm:smoothFunctionsApprox}, every function $f \in C^{s}(\R^n)$ is $r$ approximable for $r = s/n$ .
\end{remark}

We now present the main theorems of this section for the strong and weak formulation of standard linear transport equations below. Afterward, in Remark \ref{rem:WhySuperior}, we discuss to what extent the resulting approximation rates improve upon a direct application of Theorem \ref{thm:smoothFunctionsApprox} to the solution $u$ of a linear transport equation. We present the proofs of the theorems at the end of this subsection.

\begin{theorem}\label{thm:mainresult}
Let $V$ satisfy assumptions (H1) and (H2) for $k, n, D \in \N$, and $T>0$. Further let, for $r >0$, $u_0\in C^1(\R^n)$ be $r$-approximable by NNs. Let $u \in C^{1}([0,T]\times\R^n\cross [0,1]^D)$ denote the unique solution of the Cauchy problem for the parametric linear transport equation
\begin{align*}
\begin{cases}
\del_t u(t,x,\eta) + V(t,x,\eta)\cdot \grad_x u(t,x,\eta) = 0, \\
u(0,x,\eta) = u_0(x).
\end{cases}
\end{align*}
Then, for every $\eps \in (0,1)$ and every compact subset $K \subset \R^n$, there exists a NN 
$\Phi^{\overline{u},\eps}$ with $d$-dimensional input, where $d\coloneqq 1+n+D$, such that for the 
restriction $\overline{u}\coloneqq \restr{u}{[0,T]\cross K \cross [0,1]^D}$ there holds 
that, for $c = c(n,r,d,k,K,T,\|V\|_{C^k}, u_0) > 0$, 
\begin{itemize}
\item[(i)] $L\left(\Phi^{\overline{u},\eps}\right)\leq c\cdot (\ln(1/\eps)+1)$,
\item[(ii)]
$
W(\Phi^{\overline{u},\eps}) \leq c \cdot \left(\eps^{-1/r} + \eps^{-d/k}\right)\cdot (\ln(1/\eps)+1),
$
\item[(iii)]$\left\|\overline{u}-\Realization\left(\Phi^{\overline{u},\eps}\right)\right\|_{L^\infty([0,T]\cross K \cross [0,1]^D)} < \eps$,
\end{itemize}
\end{theorem}

In Theorem \ref{thm:mainresult} above, the initial condition is required to be continuously differentiable. In the following result, we extend Theorem \ref{thm:mainresult} to initial conditions that are Lipschitz continuous only. To handle initial conditions that are not continuously differentiable, we have to consider weak solutions as described in Definition $\ref{def:weaksol}$.

\begin{theorem}\label{thm:mainWeak}
Let $V$ satisfy assumptions (H1) and (H2) for $k, n, D \in \N$, and $T>0$. Further let, for $r >0$, $u_0\in W^{1, \infty}_{\loc}$ be $r$-approximable by NNs. Let $u(t,x,\eta)=u_0(X(0,t,x,\eta))$ denote the weak solution of the Cauchy problem for the parametric linear transport equation of \eqref{eq:transportequation} according to Proposition \ref{prop:weak}.

Then, for every $\eps \in (0,1)$ and every compact subset $K \subset \R^n$, there exists a NN $\Phi^{\overline{u},\eps}$ with $d$-dimensional input, where $d \coloneqq 1+n+D$, such that for the restriction $\overline{u}\coloneqq \restr{u}{[0,T]\cross K \cross [0,1]^D}$ there holds that 
\begin{itemize}
\item[(i)] $L\left(\Phi^{\overline{u},\eps}\right)\leq c \cdot \left(\ln(1/\eps)+1\right)$, 
\item[(ii)]$W\left(\Phi^{\overline{u},\eps}\right) \leq c \cdot \left(\eps^{-1/r} + \eps^{-d/k}\right)\cdot \left(\ln(1/\eps)+1\right)$,
\item[(iii)] $\left\|\overline{u}-\Realization\left(\Phi^{\overline{u},\eps}\right)\right\|_{L^\infty([0,T]\cross K \cross [0,1]^D)}< \eps$,
\end{itemize}
for $c = c(n,r,d,k,|K|,T,\|V\|_{C^k}, u_0) > 0$.
\end{theorem}

\begin{remark} \label{rem:WhySuperior}
The typical framework in which we expect to apply Theorems \ref{thm:mainresult} and \ref{thm:mainWeak} above is that where $V$ is substantially smoother than the initial condition $u_0$. Concretely, in the following two situations we have that the approximation rates resulting from an application of Theorems \ref{thm:mainresult} or \ref{thm:mainWeak} are significantly better than those resulting from a direct approximation of $u$ by Theorem~\ref{thm:smoothFunctionsApprox}.
\begin{itemize}
    \item Assume that, in the notation of Theorems \ref{thm:mainresult} and \ref{thm:mainWeak}, $n \leq s \ll d \leq k$, and $u_0 \in C^s$. In this situation, the dimension of the parameter space is significantly larger than the dimension of the physical domain. The dependence of $V$ on the parameters is, however, very regular.
    
    Then $u \in C^s([0,T]\cross K \cross [0,1]^D)$ and a direct application of Theorem~\ref{thm:smoothFunctionsApprox} would yield an approximating network with a complexity bound for the number of weights of the form $c\cdot ((\ln(1/\eps)+1)\eps^{-d/s})$. On the other hand, Theorem~\ref{thm:smoothFunctionsApprox} yields that $u_0$ is $1$-approximable by NNs and hence Theorem \ref{thm:mainresult} yields a complexity bound that is not worse than $c\cdot ((\ln(1/\eps)+1) \epsilon^{-1}$.
    
    \item Assume that $D \in \N$ and in the notation of Theorems \ref{thm:mainresult} and \ref{thm:mainWeak} $n \ll n + D$, and $k = d$. Moreover, assume that $u_0 \not \in W^{2, \infty}_{\loc}$, but $u_0 \in W^{1, \infty}_{\loc}$ and $u_0$ can be very efficiently represented by the realisation of a NN. A typical example is that $u_0$ is a ramp function along a hyperplane or a piecewise affine function. In this case, $u_0$ is $r$-approximable for every $r \in \R$. 
    
    Then we have that $u \not \in W^{2, \infty}([0,T]\cross K \cross [0,1]^D)$ and therefore a direct application of Theorem~\ref{thm:smoothFunctionsApprox} to approximate $u$ would again yield a NN with a complexity bound for the number of weights of the form $c\cdot ((\ln(1/\eps)+1)\eps^{-d})$. On the other hand, Theorem \ref{thm:mainWeak} yields 
    a complexity bound of $c\cdot ((\ln(1/\eps)+1)\eps^{-1})$.
\end{itemize}
\end{remark}

\begin{proof}[Proof of Theorem \ref{thm:mainresult}]
Recall that, by Theorem~\ref{thm:solutionTE}, the unique solution $u \in C^{1}([0,T]\times\R^n\cross [0,1]^D)$ of \eqref{eq:transportequation} is given by 
\begin{align*}
    u(t,x,\eta) = u_0(X(0,t,x,\eta)). 
\end{align*}
Moreover, $X \in C^k([0,T]\times[0,T]\times\R^n\times[0,1]^D)$ by Theorem~\ref{thm:ExistenceX}. Therefore, $\widetilde{X} \coloneqq X(0, \cdot, \cdot, \cdot) \in C^k([0,T]\times\R^n\times[0,1]^D)$.

The idea of the proof is to first approximate the functions $u_0$ and $\widetilde{X}$ separately by realisations of NNs using Theorem~\ref{thm:smoothFunctionsApprox} and then to concatenate these NNs by Proposition~\ref{prop:SparseConc}. 
To apply Theorem~\ref{thm:smoothFunctionsApprox}, we restrict the function $\widetilde{X}$ to $U \coloneqq [0,T] \cross K \cross [0,1]^D$ and $u_0$ to $B_G(0)$ where $B_G(0)\subset \R^n$ denotes the ball of radius $G$ around $0$ with $G = (|K|+CT) \exp(CT)$ from  \eqref{eq:XC0}.
Then Definition \ref{def:approximability} and Theorem~\ref{thm:smoothFunctionsApprox} imply that there exist NNs $\Phi^{u_0,\delta_1}$ and $\Phi^{\widetilde{X},\delta_2}$ with $n$ and $d$ dimensional input dimension, respectively, such that for $\delta_1\coloneqq \eps /2$ and $\delta_2\coloneqq \eps /(2 \, \mathrm{Lip}_{u_0})$ there holds
\begin{align*}
    \left\|u_0-\Realization\left(\Phi^{u_0,\delta_1}\right)\right\|_{L^\infty(B_G(0))} < \delta_1 \quad\quad  \text{ and } \quad\quad \left\|\widetilde{X}-\Realization\left(\Phi^{\widetilde{X},\delta_2}\right)\right\|_{L^\infty(U)} < \delta_2.
\end{align*}
Invoking the triangle inequality, we conclude for the concatenated network $\Phi^{u,\eps} \coloneqq \Phi^{u_0,\delta_1} \odot \Phi^{\widetilde{X}, \delta_2}$ that 
\begin{align*}
    \left\|u-\Realization\left(\Phi^{u,\eps}\right)\right\|_{L^\infty(U)} &= \left\|u_0 \circ \widetilde{X} - \Realization\left(\Phi^{u_0,\delta_1} \odot \Phi^{\widetilde{X}, \delta_2}\right)\right\|_{L^\infty(U)}\\
    &= \left\|u_0 \circ \widetilde{X} - \Realization\left(\Phi^{u_0,\delta_1}\right) \circ \Realization\left(\Phi^{\widetilde{X}, \delta_2}\right)\right\|_{L^\infty(U)} \\
    &\leq\left\|u_0 \circ \widetilde{X} - u_0 \circ \Realization\left(\Phi^{\widetilde{X},\delta_2}\right)\right\|_{L^\infty(U)} \nonumber \\
    &\qquad + \left\| u_0 \circ \Realization\left(\Phi^{\widetilde{X},\delta_2}\right) - \Realization\left(\Phi^{u_0,\delta_1}\right) \circ \Realization\left(\Phi^{\widetilde{X}, \delta_2}\right)\right\|_{L^\infty(U)}\\
    &\leq \mathrm{Lip}_{u_0} \left\|\widetilde{X}-\Realization\left(\Phi^{\widetilde{X},\delta_2}\right)\right\|_{L^\infty(U)} + \left\|u_0 - \Realization\left(\Phi^{u_0,\delta_1}\right)\right\|_{L^\infty(B_G(0))} \\
    &\leq \frac{\eps}{2} + \frac{\eps}{2} = \eps.
\end{align*}
Additionally, we compute the number of weights of $\Phi^{u,\eps}$ using Proposition \ref{prop:SparseConc}, Definition \ref{def:approximability}, Theorem~\ref{thm:smoothFunctionsApprox}, Remark \ref{rem:smoothFunctionsApprox}, and Proposition \ref{prop:boundX} as 
\begin{align*}
    W\left(\Phi^{u,\eps}\right) \leq 2 \cdot \left(W\left(\Phi^{u_0,\delta_1}\right) + W\left(\Phi^{X,\delta_2}\right)\right) \leq c \cdot \left(\eps^{-1/r} + \eps^{-d/k}\right) \cdot \left(\ln\left(1/\eps\right)+1\right)
\end{align*}
with $c = c(n,r,d,k,|K|,T,\|V\|_{C^k},u_0) > 0$. Moreover, by Proposition \ref{prop:SparseConc}, Definition \ref{def:approximability}, Theorem~\ref{thm:smoothFunctionsApprox}, Remark \ref{rem:smoothFunctionsApprox}, and Proposition \ref{prop:boundX}, we have that 
$L(\Phi^{u,\eps}) \leq c\cdot (\ln(1/\eps)+1)$.
\end{proof}

\begin{proof}[Proof of Theorem \ref{thm:mainWeak}]
The proof is completely analogous to the proof of Theorem~\ref{thm:mainresult} since the form of the solution did not change and our estimates only required $u_0$ to be Lipschitz continuous, but not that $u_0$ was in $C^1$.
\end{proof}

\subsection{Non-vanishing source term}\label{sec:source}

In the following, we extend Theorem \ref{thm:mainresult} to non-vanishing source terms. We state our results for two different types of source terms. For $V$ satisfying assumptions (H1) and (H2) with $k, n, D \in \N$, and $T>0$, we assume that one of the following properties holds:
\mathleft
\begin{align}
    \noindent &\text{   (i) } f(t,x,\eta)=f_1(t,x) \in C^{s'}([0,T]\cross\R^n),\, \quad \text{ for } s' \coloneqq \lceil (n+1) k/d \rceil\label{eq:fLow},\\
    &\text{   (ii) } f(t,x,\eta)=f_2(t,x,\eta) \in C^{s'}([0,T]\cross\R^n\cross[0,1]^D), \quad \text{ for } s' \coloneqq k. \label{eq:fHigh}
\end{align}
In words, we assume high regularity of $f$ if it depends on $\eta$ while much less regularity of $f$ is sufficient for the $\eta$-independent case.
\begin{remark}\label{rem:ApproximationOfSourceTerm}
In both cases, \eqref{eq:fLow} and \eqref{eq:fHigh}, Theorem \ref{thm:smoothFunctionsApprox} demonstrates that for every $\eps \in (0,1)$ there exists a NN $\Phi^{f, \eps}$ such that 
\begin{itemize}
    \item $L(\Phi^{f,\eps})\leq c\cdot (\ln(1/\eps)+1)$,
    \item $W(\Phi^{f,\eps})\leq c\,\eps^{-d/k}\cdot (\ln(1/\eps)+1)$,
    \item $\|f-\Realization(\Phi^{f,\eps})\|_{L^\infty} < \eps$.
\end{itemize}
\end{remark}
Based on the assumption on the source term, we next present an approximation result for solutions of non-homogeneous parametric linear transport equations. 
\mathcenter

\begin{theorem}\label{thm:resultsourceterm}
Let $V$ satisfy assumptions (H1) and (H2) for $k, n, D \in \N$, and $T>0$. Further let, for $r >0$, $u_0\in C^1(\R^n)$ be $r$-approximable by NNs and let $f$ and $s'$ be as in \eqref{eq:fLow} or \eqref{eq:fHigh}. Let $u \in C^{1}([0,T]\times\R^n\cross [0,1]^D)$
denote the unique solution of the Cauchy problem for the non-homogeneous parametric linear transport equation
\begin{align*}
\begin{cases}
\del_t u(t,x,\eta) +  V(t,x,\eta) \cdot \grad_x u(t,x,\eta) = f(t,x,\eta), \\
u(0,x,\eta)=u_0(x).
\end{cases}
\end{align*}
Then, for every $\eps \in (0,1)$ and every compact subset $K \subset \R^n$, there exists a NN $\Phi^{\overline{u},\eps}$ with $d$-dimensional input, where $d\coloneqq 1+n+D$, such that for the restriction $\overline{u} \coloneqq \restr{u}{[0,T]\cross K \cross [0,1]^D}$ there holds that
\begin{itemize}
\item[(i)] $L\left(\Phi^{\overline{u},\eps}\right)\leq c\cdot (\ln(1/\eps)+1)$, 
\item[(ii)] $W\left(\Phi^{\overline{u},\eps}\right)\! \leq \!c \cdot \left(\eps^{-1/r} + \eps^{-(d+1)/k-1}\right) \cdot \left(\ln(1/\eps)+1\right)$,
\item[(iii)]
$\left\|\overline{u}-\Realization\left(\Phi^{\overline{u},\eps}\right)\right\|_{L^\infty([0,T]\cross K \cross [0,1]^D)} < \eps$,
\end{itemize}
for $c=c(n,r,d,k,|K|,T,\|V\|_{C^k}, u_0, \|f\|_{C^{s'}}) > 0$.
\end{theorem}
\begin{remark}
\begin{itemize}
    \item If $u_0 \in C^{s}(\R^n)$, then Remark \ref{rem:RApproximable} demonstrates that we can replace $r$ in Theorem \ref{thm:resultsourceterm} by $s/n$ and the constant $c$ in Theorem \ref{thm:resultsourceterm} depends more specifically on $\|u_0\|_{C^s}$.
    \item As for Theorem \ref{thm:mainresult}, one can immediately generalise Theorem \ref{thm:resultsourceterm} to represent solutions of weak formulations via Remark \ref{rem:WeakFormulationSource}. 
\end{itemize}
\end{remark}
\begin{proof}[Proof of Theorem \ref{thm:resultsourceterm}]
Recall that, by Proposition \ref{prop:source}, the unique solution $u \in C^{1}([0,T]\times\R^n\cross [0,1]^D)$ of \eqref{eq:sourcetermeq} is given by 
\begin{align}\label{eq:OurTarget}
    u(t,x,\eta) = u_0(X(0,t,x,\eta)) + \int_0^t f(\tau,X(\tau,t,x,\eta),\eta) \dtau,
\end{align}
for all $(t,x,\eta) \in [0,T]\times\R^n\cross [0,1]^D$.

The proof, therefore, proceeds as follows: First, via Proposition \ref{prop:ApproxOfTrapez}, we construct a NN the realisation of which approximates the antiderivative of $f$ with respect to the first coordinate by a NN. Then we construct a second NN via Theorem \ref{thm:mainresult} the realisation of which approximates $u_0\circ X$. Finally, Proposition \ref{def:addofNN} yields a NN such that the associated realisation approximates \eqref{eq:OurTarget}.

Concretely, let $G = G(k,d,T,|K|,\|V\|_{C^k} ) > 0$ be as in 
Proposition \ref{prop:boundX}. Definition \ref{def:approximability} and
Theorem~\ref{thm:smoothFunctionsApprox} imply that there exist NNs 
$\Phi^{u_0,\delta_1}$, $\Phi^{X,\delta_2}$ and $\Phi^{f,\delta_3}$ with $n$, $d+1$, and $d$-dimensional input respectively such that, for \begin{align}\label{eq:definitionsOfDelta}
    \delta_1\coloneqq \eps /6, \quad \delta_2\coloneqq \eps /(12 \, \max\{\mathrm{Lip}_{u_0},\mathrm{Lip}_{f}\}), \quad  \delta_3 \coloneqq \eps/12,
\end{align}
we have that
\begin{align*}
    \left\|u_0-\Realization\left(\Phi^{u_0,\delta_1}\right)\right\|_{L^\infty(B_G(0))} < \delta_1,  \quad \left\|X-\Realization\left(\Phi^{X,\delta_2}\right)\right\|_{L^\infty([0,T]\times [0,T]\times K\cross [0,1]^D)} &< \delta_2, \quad  \text{ and } \\ \left\|f-\Realization\left(\Phi^{f,\delta_3}\right)\right\|_{L^\infty([0,T]\times B_G(0) \cross [0,1]^D)} &< \delta_3 .
\end{align*}

Let $\Phi_\tau \coloneqq \left(\left(\bmat{cccc}{1 & 0 & \dots& 0}, 0\right)\right)$ be a NN with one layer and input dimension $2 + n + D$. Moreover, let 
$$
\Phi_\eta \coloneqq \left(\left(\left[ 0_{\R^{D \times (2+n)}} \big\vert \mathrm{Id}_{\R^{D}}\right], 0_{\R^D}\right)\right).
$$
We have that $\Realization(\Phi_\tau)(\tau, t, x, \eta) = \tau$ and $\Realization(\Phi_\eta)(\tau, t, x, \eta) = \eta$, for all $(\tau, t, x, \eta) \in [0,T]\times [0,T]\times K\cross [0,1]^D$. Setting $\Phi^{X, \delta_2, \mathrm{full}} \coloneqq \mathrm{P}(\Phi_\tau, \Phi^{X,\delta_2}, \Phi_\eta)$, we now have that 
$$
\sup_{(\tau, t, x, \eta) \in [0,T]\times [0,T]\times K\cross [0,1]^D}\left| \left(\tau, X\left(\tau, t, x, \eta\right), \eta\right) - \Realization\left(\Phi^{X, \delta_2, \mathrm{full}}\right) \right| \leq \delta_2.
$$
Denoting $f^X(\tau, t, x, \eta) \coloneqq f(\tau,X(\tau,t,x,\eta),\eta)$, we have by the triangle inequality that 
\begin{align}
    &\sup_{(\tau, t, x, \eta) \in [0,T]\times [0,T]\times K\cross [0,1]^D} \left| f^X(\tau,t,x,\eta)  - \Realization\left(\Phi^{f, \delta_3} \odot \Phi^{X, \delta_2, \mathrm{full}}\right)(\tau,t,x,\eta)\right|\nonumber\\
    &\leq \sup_{(\tau, t, x, \eta) \in [0,T]\times [0,T]\times K\cross [0,1]^D} \left| f^X(\tau,t,x,\eta)  - f \circ \Realization\left(\Phi^{X, \delta_2, \mathrm{full}}\right)(\tau,t,x,\eta)\right|\nonumber\\
    &\qquad + \sup_{(\tau, t, x, \eta) \in [0,T]\times [0,T]\times K\cross [0,1]^D} \left|f\circ \Realization\left(\Phi^{X, \delta_2, \mathrm{full}}\right)(\tau,t,x,\eta) - \Realization\left(\Phi^{f, \delta_3} \odot \Phi^{X, \delta_2, \mathrm{full}}\right)(\tau,t,x,\eta)\right|\nonumber\\
    &\leq \mathrm{Lip}_f \delta_2 + \delta_3. \label{eq:estimateOffX}
\end{align}
Let $\widetilde{A}$ be the matrix satisfying $\widetilde{A}(t, x, \eta) = (t, t, x, \eta)$ for $(t, x, \eta) \in [0,T]\times K\cross [0,1]^D$. Then, for 
\begin{align}\label{eq:definitionOfN}
N \coloneqq \left\lceil\frac{15}{\eps} \max\left\{\|f^X\|_{C^1}, 1 + \|f\|_{L^{\infty}([0,1] \times B_G(0) \times [0,1]^D)}\right\}\right\rceil,
\end{align}
we define with Proposition \ref{prop:ApproxOfTrapez}
$$
\Phi^{f, \mathrm{anti}} \coloneqq \widetilde{I}_N\left(\Phi^{f, \delta_3} \odot \Phi^{X, \delta_2, \mathrm{full}}\right) \sconc \left( \left( \widetilde{A}, 0 \right) \right).
$$
We have the following estimate: 
\begin{align}
    &\left| \int_0^t f(\tau,X(\tau,t,x,\eta),\eta) \dtau  - \Realization\left(\Phi^{f, \mathrm{anti}}\right)(t, x, \eta)\right| \nonumber\\
    &\!\!\!\!\!\!\!\!\underset{\text{\footnotesize{Prop. \ref{prop:trapez}}}}{
    \leq} \left| I_N(f^X)(t,t,x,\eta)  - \Realization\left(\Phi^{f, \mathrm{anti}}\right)(t, x, \eta)\right| + \frac{2 \|f^X\|_{C^1}}{N}\nonumber\\
    & \leq \left| I_N(f^X)(t,t,x,\eta)  - \Realization\left( \widetilde{I}_N\left(\Phi^{f, \delta_3} \odot \Phi^{X, \delta_2, \mathrm{full}}\right)\right)(t, t, x, \eta)\right| + \frac{2 \|f^X\|_{C^1}}{N}\nonumber\\
    &\!\!\!\!\!\!\!\!\underset{\text{\footnotesize{ Rem.  \ref{rmk:TrapRule}}}}{
    \leq}\left| I_N\left(f^X\right)(t,t,x,\eta)  - I_N\left(\Realization\left(\Phi^{f, \delta_3} \odot \Phi^{X, \delta_2, \mathrm{full}})\right)\right)(t, t, x, \eta)\right|\nonumber\\
    &\qquad + \frac{2 \|f^X\|_{C^1}}{N} + \frac{3 \left\|\Realization\left(\Phi^{f, \delta_3} \sconc \Phi^{X, \delta_2, \mathrm{full}}\right)\right\|_{L^\infty}}{N}\nonumber\\
    &
    \leq\left| I_N\left(f^X\right)(t,t,x,\eta)  - I_N\left(\Realization\left(\Phi^{f, \delta_3} \odot \Phi^{X, \delta_2, \mathrm{full}}\right)\right)(t, t, x, \eta)\right| + \frac{2 \|f^X\|_{C^1}}{N} + \frac{3 \|\Realization\left(\Phi^{f, \delta_3}\right)\|_{L^\infty}}{N}.\nonumber
\end{align}
By elementary estimates, we therefore conclude that 
\begin{align}
    &\left| \int_0^t f(\tau,X(\tau,t,x,\eta),\eta) \dtau  - \Realization\left(\Phi^{f, \mathrm{anti}}\right)(t, x, \eta)\right|\nonumber\\
    &\leq \sup_{(\tau, t, x, \eta) \in [0,T]\times [0,T]\times K\cross [0,1]^D)} \left| f^X(\tau,t,x,\eta)  - \Realization\left(\Phi^{f, \delta_3} \odot \Phi^{X, \delta_2, \mathrm{full}}\right)(\tau,t,x,\eta)\right| + \nonumber\\
    &\qquad + \frac{2 \|f^X\|_{C^1} + 3 \|\Realization\left(\Phi^{f, \delta_3}\right)\|_{L^\infty}}{N}\nonumber\\
    &\!\!\!\underset{\text{\footnotesize{\eqref{eq:estimateOffX}}}}{
    \leq} \mathrm{Lip}_f \delta_2 + \delta_3 + 
    \frac{2 \|f^X\|_{C^1} + 3 \|\Realization\left(\Phi^{f, \delta_3}\right)\|_{L^\infty}}{N}. \label{eq:EstimateOfIntegralPart}
\end{align}

Setting $\widetilde{X} \coloneqq X(0, \cdot, \cdot, \cdot)$, we conclude for the NN $\Phi^{u,\eps} \coloneqq (\Phi^{u_0,\delta_1} \odot \Phi^{\widetilde{X}, \delta_2})\oplus \Phi^{f, \mathrm{anti}}$ that 
\begin{align*}
    &\left\|u-\Realization\left(\Phi^{u,\eps}\right)\right\|_{L^\infty([0,T]\times K\cross [0,1]^D)}\\
    &\qquad = \left\|u_0 \circ \widetilde{X} \, +  \int_0^t f(\tau,X(\tau,t,x,\eta),\eta) \dtau - \Realization\left(\Phi^{u_0,\delta_1} \odot \Phi^{\widetilde{X}, \delta_2}\right) -  \Realization\left(\Phi^{f, \mathrm{anti}}\right)\right\|_{L^\infty([0,T]\times K\cross [0,1]^D)} \\
    &\qquad \!\!\!\underset{\text{\footnotesize{\eqref{eq:EstimateOfIntegralPart}}}}{
    \leq} \left\|u_0 \circ \widetilde{X} - \Realization(\Phi^{u_0,\delta_1}) \circ \Realization(\Phi^{\widetilde{X}, \delta_2})\right\|_{L^\infty([0,T]\times K\cross [0,1]^D)}\\
    & \qquad \qquad + \mathrm{Lip}_f \delta_2 + \delta_3 + 
    \frac{2 \|f^X\|_{C^1} + 3 \|\Realization\left(\Phi^{f, \delta_3}\right)\|_{L^\infty}}{N} \\
    &\qquad \leq \delta_1 + \mathrm{Lip}_{u_0} \delta_2 + \mathrm{Lip}_f \delta_2 + \delta_3 + 
    \frac{2 \|f^X\|_{C^1} + 3 \|\Realization\left(\Phi^{f, \delta_3}\right)\|_{L^\infty}}{N}\\
    &\qquad \leq \frac{\eps}{6} + \frac{\eps}{12} + \frac{\eps}{12} + \frac{\eps}{12} +
\frac{\eps}{3} < \eps,\end{align*}
where we have used \eqref{eq:definitionOfN} and \eqref{eq:definitionsOfDelta}.

We compute the number of weights using Definition \ref{def:approximability} Theorem~\ref{thm:smoothFunctionsApprox}, Propositions \ref{prop:SparseConc}, \ref{prop:ApproxOfTrapez}, and Remark \ref{rem:ApproximationOfSourceTerm} as 
\begin{align*}
    W\left(\Phi^{u,\eps}\right) 
    &\leq  W\left(\Phi^{u_0,\delta_1} \odot \Phi^{X,\delta_2}\right) +  W\left(\Phi^{f, \mathrm{anti}}\right) \\
    & \leq c \cdot (\eps^{-1/r} + \eps^{-d/k})\cdot (\ln(1/\eps)+1) + c'\cdot N \cdot W\left(\Phi^{f, \delta_3} \odot \Phi^{X, \delta_2, \mathrm{full}} \right) \\
    & \leq c \cdot \left(\eps^{-1/r} + \eps^{-d/k}\right)\cdot \left(\ln(1/\eps)+1\right) + c''\cdot \eps^{-1} \cdot \left(\eps^{-d/k}  + \eps^{-(d+1)/k}\right) \\
    & \leq c''' \cdot (\eps^{-1/r} + \eps^{-(d+1)/k-1}) \cdot (\ln(1/\eps)+1)
\end{align*}
with $c'''=c'''(n,r,d,k,|K|,T,\|V\|_{C^k},u_0, \|f\|_{C^{s'}})>0$. The number of layers is $c''''\cdot (\ln(1/\eps)+1)$, for $c''''= c''''(n,r,d,k,|K|,T,\|V\|_{C^k},u_0, \|f\|_{C^{s'}})>0$, by the same results.
\end{proof}

\subsection{Conservative form}
Next, we extend our results to transport equations in conservative form by invoking Proposition \ref{prop:cons}.
\begin{theorem}\label{thm:resultcons}
Let $V$ satisfy assumptions (H1) and (H2) with $n,D,k\in \N$, and $T>0$. Further let, for $r >0$, $u_0\in C^1(\R^n)$ be $r$-approximable by NNs. Let $u \in C^{1}([0,T]\times\R^n\cross [0,1]^D)$
denote the unique solution of the Cauchy problem for the conservative parametric linear transport equation
\begin{align*}
\begin{cases}
\del_t u(t,x,\eta) +  \mathrm{div}_x(V(t,x,\eta) u(t,x,\eta)) =0, \\
u(0,x,\eta)=u_0(x).
\end{cases}
\end{align*}
Then, for every $\eps \in (0,1)$ and every compact subset $K \subset \R^n$, there exists a NN $\Phi^{\overline{u},\eps}$ with $d$-dimensional input, where $d\coloneqq 1+n+D$, such that for the restriction $\overline{u}\coloneqq \restr{u}{[0,T]\cross K \cross [0,1]^D}$ there holds
\begin{itemize}
\item[(i)] $L\left(\Phi^{\overline{u},\eps}\right)\leq c \cdot (\ln(1/\eps)+1)$,
\item[(ii)] $W\left(\Phi^{\overline{u},\eps}\right) \leq c \cdot \left(1+ \eps^{-1/r} + \eps^{-d/(k-1)} \right)\cdot \left(\ln(1/\eps)+1\right)$,
\item[(iii)] $\left\|\overline{u}-\Realization\left(\Phi^{\overline{u},\eps}\right)\right\|_{L^\infty([0,T] \times K \times [0,1]^D)} < \eps$,
\end{itemize}
for $c = c(n,r,d,k,|K|,T,\|V\|_{C^k}, u_0) > 0$.
\end{theorem}
\begin{remark}
\begin{itemize}
    \item If $u_0 \in C^{s}(\R^n)$, then Remark \ref{rem:RApproximable} demonstrates that we can replace $r$ in Theorem \ref{thm:resultcons} by $s/n$ and the constant $c$ in Theorem \ref{thm:resultcons} depends more specifically on $\|u_0\|_{C^s}$.
    \item As in earlier results, the statement of Theorem \ref{thm:resultcons}, extends to a weak formulation of the transport equation via Remark \ref{rem:WeakSolutionConservative}.
\end{itemize}

\end{remark}

\begin{proof}
Recall that, by Proposition \ref{prop:cons}, the unique solution $u \in C^{1}([0,T]\times\R^n\cross [0,1]^D)$ of \eqref{eq:transportequation} is given by 
\begin{align*}
    u(t,x,\eta) = u_0(X(0,t,x,\eta))J(0,t,x,\eta). 
\end{align*}
Moreover, $X \in C^k([0,T]\cross [0,T]\times\R^n\times[0,1]^D)$ by Theorem~\ref{thm:ExistenceX} and therefore $J \in C^{k-1}([0,T]\cross[0,T]\times\R^n\times[0,T]^D)$. 

The proof proceeds by first approximating $J$ by a NN with help of Theorem~\ref{thm:smoothFunctionsApprox}, then invoking the known approximation of $u_0\circ X$ via Theorem \ref{thm:mainresult}, and then applying the multiplication of NNs by Proposition \ref{prop:multiplication}.

Let $G_1 \coloneqq G_0 + 3 \|V\|_{C^1}\exp(T\|V\|_{C^1})$ be the bound of  $\|X\|_{C^1}$ from \eqref{boundXC1}. The Hadamard inequality \cite[Corollary 7.8.2]{Horn1985} implies that
\begin{align*}
    \|J\|_{L^\infty([0,T] \times K \times [0,1]^D)} &\leq \sup_{(t,x,\eta) \in [0,T] \times K \times [0,1]^D} \prod_{i=1}^n \left\|\left(D_x X(0,t,x,\eta)\right)_i\right\|_2 \\
    &\leq \prod_{i=1}^n \sqrt{n \|X\|_{C^1}^2} = n^{n/2} \|X\|_{C^1}^n \leq n^{n/2} G_1^n  \eqqcolon G_J,
\end{align*} 
where $(D_x X(s,t,x,\eta))_i$ denotes the $i$-th column vector of the matrix $D_x X$.\\ 
\mbox{Theorem~\ref{thm:smoothFunctionsApprox}} implies that there exist NNs $\Phi^{u_0,\delta_1}$, $\Phi^{X,\delta_2}$, and
$\Phi^{J,\delta_3}$ with $n$, $d$, and $d$-dimensional input dimension
respectively such that, for $\delta_1\coloneqq \eps /(8 G_J)$, 
$\delta_2\coloneqq \eps /(8 \, \mathrm{Lip}_{u_0} G_J)$, and $\delta_3 \coloneqq \eps/(4 \|u_0\|_{L^\infty})$, there holds
\begin{align*}
    \left\|u_0-\Realization\left(\Phi^{u_0,\delta_1}\right)\right\|_{L^\infty}< \delta_1,  \quad \left\|X-\Realization\left(\Phi^{X,\delta_2}\right)\right\|_{L^\infty} < \delta_2,\quad  \text{ and } \quad\left\|J-\Realization\left(\Phi^{J,\delta_3}\right)\right\|_{L^\infty} < \delta_3 .
\end{align*}
We conclude for $\Phi^{u,\eps} \coloneqq (\Phi^{u_0,\delta_1} \odot \Phi^{X, \delta_2})\otimes^{\eps/4} \Phi^{J, \delta_3} $ that 
\begin{align*}
    &\left\|u-\Realization\left(\Phi^{u,\eps}\right)\right\|_{L^\infty([0,T] \times K \times [0,1]^D)}\\
    &\qquad = \left\|\left(u_0 \circ X\right) \cdot J - \Realization\left(\left(\Phi^{u_0,\delta_1} \odot \Phi^{X, \delta_2}\right)\otimes^{\eps/4} \Phi^{J, \delta_3}\right)\right\|_{L^\infty([0,T] \times K \times [0,1]^D)} \\
    &\qquad= \left\| \left( u_0 \circ X \right) \cdot J - \left(\Realization\left(\Phi^{u_0,\delta_1}\right) \circ \Realization\left(\Phi^{X, \delta_2}\right)\right) \cdot  \Realization\left(\Phi^{J, \delta_3}\right) \right\|_{L^\infty([0,T] \times K \times [0,1]^D)}+ \frac{\eps}{4} \\
    &\qquad\leq \left\| \left( u_0 \circ X \right) \cdot J - \Realization\left(\Phi^{u_0,\delta_1}\right)\circ \Realization\left(\Phi^{X,\delta_2}\right)\cdot J\right \|_{L^\infty([0,T] \times K \times [0,1]^D)} \nonumber\\
    &\qquad \qquad + \left\| \left( \Realization\left(\Phi^{u_0,\delta_1}\right)\circ \Realization\left( \Phi^{X,\delta_2}\right)\right) \cdot J - \left(\Realization\left(\Phi^{u_0,\delta_1}\right) \circ \Realization\left(\Phi^{X, \delta_2}\right) \right) \Realization\left(\Phi^{J, \delta_3}\right)\right\|_{L^\infty([0,T] \times K \times [0,1]^D)} + \frac{\eps}{4}\\
    &\qquad\leq  \left\| u_0 \circ X -\Realization\left(\Phi^{u_0,\delta_1}\right) \circ \Realization\left(\Phi^{X,\delta_2}\right)\right\|_{L^\infty([0,T] \times K \times [0,1]^D)} \|J\|_{L^\infty([0,T] \times K \times [0,1]^D)} \\
    &\qquad \qquad + \left\|\Realization\left(\Phi^{u_0,\delta_1}\right) \circ \Realization\left(\Phi^{X, \delta_2}\right) \right\|_{L^\infty([0,T] \times K \times [0,1]^D)} \left\|J - \Realization\left(\Phi^{J, \delta_3}\right)\right\|_{L^\infty([0,T] \times K \times [0,1]^D)} + \frac{\eps}{4} \\
    &\qquad\leq \frac{\eps}{4} + \left\|u_0 \circ X -\Realization\left(\Phi^{u_0,\delta_1}\right) \circ \Realization\left(\Phi^{X,\delta_2}\right)\right\|_{L^\infty([0,T] \times K \times [0,1]^D)} \left\|J - \Realization\left(\Phi^{J, \delta_3}\right)\right\|_{L^\infty([0,T] \times K \times [0,1]^D)}\\
    &\qquad \qquad + \left\|u_0 \circ X\right\|_{L^\infty([0,T] \times K \times [0,1]^D)}\left\|J - \Realization\left(\Phi^{J, \delta_3}\right)\right\|_{L^\infty([0,T] \times K \times [0,1]^D)} + \frac{\eps}{4}\\  
    &\qquad \leq \frac{\eps}{4}+\frac{\eps}{4}+\frac{\eps}{4}+\frac{\eps}{4}= \eps.
\end{align*}
We have assumed without loss of generality that $\|J - \Realization(\Phi^{J, \delta_3})\|_{L^\infty([0,T] \times K \times [0,1]^D)} \leq 1$ and have used the previous result from Theorem~\eqref{thm:mainresult} to bound $\|u_0 \circ X -\Realization(\Phi^{u_0,\delta_1}) \circ \Realization(\Phi^{X,\delta_2})\|_{L^\infty([0,T] \times K \times [0,1]^D)}$.\\
We compute the number of weights with Definition \ref{def:approximability},  Theorem~\ref{thm:smoothFunctionsApprox},  Proposition \ref{prop:SparseConc}, and \ref{prop:multiplication} as 
\begin{align*}
    W\left(\Phi^{u,\eps}\right) 
    &\leq c \cdot \left(\ln(1/\eps)+1\right) +  2 \cdot \left(W\left(\Phi^{u_0,\delta_1} \odot \Phi^{X,\delta_2}\right) + W\left(\Phi^{J,\delta_3}\right)\right) \\
    & \leq c \cdot \left(1+ \eps^{-1/r} + \eps^{-d/k} + \eps^{d/{k-1}}\right) \cdot (\ln(1/\eps)+1) 
\end{align*}
with $c=c(n,r,d,k,|K|,T,\|V\|_{C^k},u_0) > 0$. The number of layers is $c \cdot (\ln(1/\eps)+1)$ by the same theorems and propositions.
\end{proof}

\section{Extensions}\label{sec:Extensions}

Below, we discuss some natural extensions of our work to more general settings.

\begin{itemize}
    \item \emph{Non-linear transport equations:} An immediate question is to what extent the results carry over to the non-linear setting. We believe that it is highly unlikely that similar results hold in this regime without overly restrictive assumptions. Indeed, in the non-linear case, non-smoothness of the initial condition $u_0$ potentially implies non-smoothness of the characteristic curves described by $X$. This can already be seen in the one-dimensional case. We consider the one-dimensional non-linear transport equation
        \begin{align*}
            \begin{cases}
            \del_t u(t,x,\eta) +  \del_x[f(u(t,x,\eta))] =0 , \\
            u(0,x,\eta)=u_0(x).
            \end{cases}
        \end{align*}
    The characteristic system of ODEs is then given by \cite[p. 26]{Serre1999} as 
    \begin{numcases}{\label{eq:charsysnonlin}}
    \del_s X(s,t,x,\eta) = f'(u(t,X(s,t,x,\eta),\eta)) , \\
    X(t,t,x,\eta)=x.
\end{numcases}
    Hence, the regularity of the characteristic curves described by $X$ depends on the global regularity of $u$ and therefore on the regularity of $u_0$ and, therefore, $X$ is not guaranteed to be smooth. 

    If $X$ is non-smooth, then the fundamental backbone of the argument, which is that $u$ can be written as the composition of a high-dimensional smooth and low-dimensional (potentially) rough function, collapses.
    
    \item \emph{Damping/amplification}: The extension of our results to parametric linear transport equations that include an amplification or damping factor is straight-forward. More precisely, we consider solutions of the equation 
    \begin{numcases}{\label{eq:amplificationeq}}
      \del_t u(t,x,\eta) + V(t,x,\eta)\cdot \grad_x u(t,x,\eta) + a(t,x,\eta) u(t,x,\eta)=0,\\
      u(0,x,\eta)=u_0(x),
    \end{numcases}
    where $a$ is, similarly to \eqref{eq:fLow} and \eqref{eq:fHigh}, either given by 
    $a(t,x,\eta)=a_1(t,x) \in C^{s'}([0,T]\cross\R^n)$, for $s' \coloneqq \lceil (n+1) k/d\rceil$, or 
     $a(t,x,\eta)=a_2(t,x,\eta) \in C^{s'}([0,T]\cross\R^n\cross[0,1]^D)$, for $s' \coloneqq k$. \\
     If $V$ satisfies assumptions (H1) and (H2) with $n,D,k\in \N$, $T>0$, and $u_0 \in C^1(\R^n)$ being $r$-approximable by NNs for $r >0$, then one can show, see \cite{LectureNotesMeanField}, that there exists a unique solution of \eqref{eq:amplificationeq} which is given by
     \begin{align}\label{eq:solamp}
         u(t,x,\eta) = u_0(X(0,t,x,\eta)) \exp\left(-\int_0^t a(\tau, X(\tau,t,x,\eta), \eta) \dtau \right).
     \end{align}
    To get an estimate on the sizes of approximating NNs for functions of the form of \eqref{eq:solamp}, we only have to combine previous results. 
    Section \ref{sec:source} describes how to approximate the map  $(t, x, \eta) \mapsto -\int_0^t a(\tau, X(\tau,t,x,\eta), \eta) \dtau$ by realisations of NNs.
    This approximation can be concatenated with an approximation of the smooth, one-dimensional exponential function via Proposition \ref{prop:SparseConc}. 
    Finally, the result may be multiplied, via Proposition \ref{prop:multiplication}, with the already known approximation of $u_0(X(0,t,x,\eta))$ from Theorem \ref{thm:mainresult}. 
    This yields a NN $\Phi^{\overline{u},\eps}$ such that the realisation of $\Phi^{\overline{u},\eps}$ approximates \eqref{eq:solamp} up to an error of $\eps>0$. 
    Estimating the individual sizes of the networks involved in the construction of $\Phi^{\overline{u},\eps}$ yields that
       \begin{align*}
        L\left(\Phi^{\overline{u},\eps}\right) &\leq c \cdot \left(\ln(1/\eps)+1\right),\\
        W\left(\Phi^{\overline{u},\eps}\right) &\leq c \cdot \left( \eps^{-1/r} + \eps^{-(d+1)/k-1}  \right)\cdot \left(\ln(1/\eps)+1\right),
        \end{align*}
    with $c = c(n,r,d,k,|K|,T,\|V\|_{C^k}, u_0, \|a\|_{C^{s'}} ) > 0$. 
    As before, if $u_0 \in C^{s}(\R^n)$, then $r = s/n$.
    \item \emph{Parameter dependence of initial condition:} We only considered the case where $u_0$ does not depend on the parameters. It is not hard to see that, in the framework of $r$-approximability, the same result would hold if $u_0$ depended on the parameters. However, if $u_0\in C^s(\R^n \times [0,1]^D)$, then Remark \ref{rem:RApproximable} would yield an approximation rate depending on the dimension $D$ of the parameter space. \\
    For an application of Remark \ref{rem:RApproximable} it is required that $u_0$ is a low-dimensional function. Hence, if $u_0 \in C^s$ depends on very few parameters, say the first $t \ll D$, then all main theorems can be extended directly. Instead of approximating $x \mapsto u_0(x)$ with a NN up to an error of $\epsilon > 0$ and having to use $\mathcal{O}(\epsilon^{-n/s})$ many weights for $\eps \to 0$, one would instead approximate $x \mapsto u_0(x, \eta_1, \dots, \eta_t)$ which requires $\mathcal{O}(\epsilon^{-(n+t)/s})$ many weights for $\eps \to 0$. \\
    A second framework in which $u_0$ could be guaranteed to be $r$-approximable with large $r$ while having low spatial smoothness is that where the parameter dependence is decoupled from the dependence on the spatial coordinates. For example, if $u_0(x, \eta) = \tilde{u}_0(x) \cdot \kappa_1(\eta) + \kappa_2(\eta)$ for smooth $\kappa_1, \kappa_2$, then again low regularity of $\tilde{u}_0$ could suffice to achieve fast rates.
    
    \item \emph{Weak solutions with discontinuous initial condition:} Since realisations of deep neural networks are always continuous functions, we cannot hope to obtain approximation results in the uniform norm as studied in this work. However, if one considers $L^p$-approximation, for $p \in [1,\infty)$ instead, then approximation of piecewise regular functions is possible. This situation was studied in \cite{petersen2018optimal}. 
\end{itemize}

\section*{Acknowledgements}
The authors would like to thank Avi Mayorcas for inspiring discussions in the early stage of this work. 
P.P is grateful for the hospitality and support of the Institute of Mathematics of the University of Oxford during his visit in January 2020.
\appendix

\section{Proof of Proposition \ref{prop:boundX}} \label{app:boundX}
\begin{proof}
For simplicity we drop the $\eta$-dependence of $X$ since derivatives with respect to $\eta$ are easy to compute and bound. Moreover, we use the norm and semi-norm 
\begin{align*}
            \|X\|_{C^k(\Omega)} \coloneqq \max_{\alpha: |\alpha|\leq k} \sup_{x \in \Omega} |D^\alpha X(x)|,  \qquad 
            |X|_{C^k(\Omega)} \coloneqq \max_{\alpha: |\alpha|= k} \sup_{x \in \Omega} |D^\alpha X(x)|,
\end{align*}
write $\|X\|_{C^{j}}$ for $ \|X\|_{C^{j}([0,T]\cross [0,T]\cross K)}$, and define $d\coloneqq 1+1+n$. We start with the definition of $X$ given by 
\begin{numcases}{\label{Xdef}}
\del_s X(s,t,x)=V(s,X(s,t,x)),\label{Xdef1}\\
X(t,t,x)=x.
\end{numcases}
By the fundamental theorem of calculus we conclude 
\begin{align}\label{eq:formulaX}
    X(s,t,x)=x + \int^s_t V(\tau,X(\tau,t,x)) \dtau.
\end{align}
In the following, we show how to bound the zeroth, first, and second derivatives to identify the pattern how these bounds are built.
With the help of the sub-linear growth-condition (H2) we conclude
\begin{align*}
    |X(s,t,x,\eta)|\leq |x| + \Big|\int_t^s |V(\tau,X(\tau,t,x,\eta),\eta)|\,\mathrm{d} \tau \Big| &\leq |x| + C \Big |  \int_t^s (1+|X(\tau,t,x,\eta)|) \,\mathrm{d} \tau\Big| \\
    &\leq |x| + CT + C \int_t^s |X(\tau,t,x,\eta)| \, \mathrm{d} \tau.
\end{align*}
and by Gronwall's inequality 
\begin{align*}
    \sup_{s \in [0,T]} |X(s,t,x)| \leq (|x| + CT) \exp(C T).
\end{align*}
Hence, 
\begin{align}\label{eq:XC0}
    \|X\|_{C^{0}} \leq (|K| + CT) \exp(C T)\eqqcolon G_0(T,|K|,\|V\|_{C^0})\eqqcolon G_0.
\end{align}
Now we derive bounds for the first derivatives. In the following, we abbreviate $\|V\|_{C^j([0,T]\cross G_0)}$ by $\|V\|_{C^j}$. We have by \eqref{Xdef1} 
\begin{align*}
    \|\del_s X\|_{C^0} \leq \|V\|_{C^0}.
\end{align*}
Furthermore, by Leibniz integral rule applied to \eqref{eq:formulaX}
\begin{align*}
     \del_t X(s,t,x) = -V(t,x) + \int^s_t \grad_x V(\tau, X(\tau,t,x)) \del_t X(\tau,t,x) \dtau
\end{align*}
and therefore 
\begin{align*}
    |\del_t X(s,t,x)| \leq \|V\|_{C^0} + \|V\|_{C^1} \int^s_t |\del_t X(\tau,t,x)| \dtau. 
\end{align*}
Gronwall's inequality implies then
\begin{align*}
    \|\del_t X\|_{C^0} \leq \|V\|_{C^0} \exp(T \|V\|_{C^1}). 
\end{align*}
The same procedure results for $\grad_x X$ in 
\begin{align*}
    \|\grad_x X\|_{C^0} \leq \exp(T \|V\|_{C^1}).
\end{align*}
Thus, we get after assuming without loss of generality that $T\geq 1, \|V\|_{C^j}\geq 1$
\begin{align}\label{boundXC1}
    \|X\|_{C^1}\leq \max\left\{G_0, \|V\|_{C^1} \exp (T \|V\|_{C^1})\right\}.
\end{align}
As the next step, we derive a bound for the second derivatives. We have
\begin{align*}
    \del_{ss} X(s,t,x)= \del_s V(s,X(s,t,x)) + \nabla_x V(s,X(s,t,x)) \del_s X(s,t,x) 
\end{align*}
and consequently
\begin{align*}
  \|\del_{ss} X\|_{C^0} \leq \|V\|_{C^1} + \|V\|^2_{C^1} \leq 2 \|V\|^2_{C^1}.
\end{align*}
Moreover, 
\begin{align*}
    \|\del_{st} X \|_{C^0}=\| \grad_x V(s,X(s,t,x)) \del_t X(s,t,x)\|_{C^0} \leq \|V\|^2_{C^1} \exp(T \|V\|_{C^1})
\end{align*}
and with
\begin{align*}
        \del_{tt} X(s,t,x)&=-\del_t V(t,x) - \grad_x V(t,X(t,t,x)) \del_t X(t,t,x)  \\
        &\qquad + \int^s_t \grad_{xx} V  (\tau,X(\tau,t,x))\del_t X(\tau,t,x)\del_t X(\tau,t,x)\dtau\\
        &\qquad \qquad + \int^s_t \grad_x V(\tau,X(\tau,t,x)) \del_{tt} X(\tau,t,x) \dtau
\end{align*}
we conclude again by Gronwall's inequality
\begin{align*}
    \|\del_{tt} X\|_{C^0} \leq 4T \|V\|^3_{C^2} \exp(T \|V\|_{C^1})^3. 
\end{align*}
The derivatives $\del_{sx}X,\del_{tx}X$ and $\del_{xx}X$ can be computed in the same way and bounded by the same term as $\del_{tt}X$. Note that the semi-norm $|X|_{C^1}$ is bounded by $|X|_{C^2}$. Thus, we receive
\begin{align*}
    \|X\|_{C^2} \leq \max\left\{G_0, 2\cdot 4 T \|V\|^3_{C^2} \exp (T \|V\|_{C^1})^3\right\}.
\end{align*}
We can iterate this procedure to receive the bound for the $k$-th derivative of the form: 
\begin{align}\label{eq:boundX}
    \|X\|_{C^k} \leq \max\left\{ G_0, 2^k T^{k-1} \|V\|^{2k-1}_{C^k} \exp(T \|V\|_{C^1})^{2k-1}\right\},
\end{align}
where we have used that the semi-norm $|\cdot|_{C^k}$ bounds the $k-1$ previous semi-norms.
The bound in \eqref{eq:boundX} includes a factor $2^k$ because every application of the product rule doubles the number of summands and a factor $T^{k-1}$ because with every derivative we get a new factor $T$ from bounding $\|\int^s_t \dots\|$.\qedhere
 \end{proof}

\textbf{Remark.} The proof of Proposition \ref{prop:boundX} might not provide a very sharp bound for the $C^k$-norm of $X$. However, for our approximation result in Chapter 4 it is only important that the norm is bounded by some constant that depends on the given data, since the norm of $X$ only enters as a constant and does not influence the asymptotic behaviour in the tolerance $\eps$. Whether the size of the bound influences the practicality of the approximation via deep NNs will be investigated in numerical simulations in future work.

\section{Construction of a NN emulating the left Riemann sum}\label{app:TrapRule}

\begin{proposition}\label{prop:ApproxOfTrapez}
Let $d \in \N_{\geq 2}$, $T>0$, $\Omega \subset \R^{d-1}$, and let $\Phi$ be a NN with $d$-dimensional input. Then there exists a NN $\widetilde{I}_N(\Phi)$ such that 
\begin{flalign}\quad &\bullet \ L\left(\widetilde{I}_N(\Phi)\right) = L(\Phi) + c_1, && \nonumber \\
\quad &\bullet \ W\left(\widetilde{I}_N(\Phi)\right) \leq c_2 \cdot N \cdot W(\Phi),&&\nonumber\\
\quad &\bullet \ \sup_{t \in [0,T], x \in \Omega}\left|\Realization\left(\widetilde{I}_N(\Phi)(t,x)\right) - \frac{1}{N}\sum_{i=0}^{\lceil t N /T \rceil-1} \Realization(\Phi)\left( \frac{i T }{N} , x\right) \right| \leq \frac{c_3}{N},&&\label{eq:ApproxStatementOfIntegralProposition}
\end{flalign}

where $c_1, c_2 >0$ are independent of $\Phi$ and $c_3 \coloneqq 3 \|\Realization(\Phi)\|_{L^\infty([0,T]\times \Omega)}$. \end{proposition}
\begin{proof}
Let, for $i \in \{ 0, 1, \dots, N\}$, $t_i \coloneqq i T /N$. 
We define, for $i\in \{0, \dots, N-1\}$,
\begin{align*}
    \Phi_i^{\mathrm{(shift)}} \coloneqq \Phi \sconc  \left( \left(\begin{array}{cc}
        0 & 0 \\
        0 & \mathrm{Id}_{\R^{d-1}}
    \end{array}\right), \left(\begin{array}{c}
         t_i  \\
         0 
    \end{array}\right) \right) .
\end{align*}
Then $\Realization(\Phi_i^{\mathrm{(shift)}})(t,x) = \Realization(\Phi)(t_i, x)$, for all $t \in [0,T], x \in \Omega$. Moreover, $W(\Phi_i^{\mathrm{(shift)}}) \leq 2W(\Phi) + 2d$ and $L(\Phi_i^{\mathrm{(shift)}}) = L(\Phi) + 2$ by Proposition \ref{prop:SparseConc}.
 Next, we define the following indicator networks for $i \in \{0, \dots, N-1\}$: 
\begin{align*}
\Phi^{\mathrm{(ind)}}_i \coloneqq \left( \left( \bmat{cccc}{1& 0& \dots& 0}, 0 \right), \left( \bmat{c|c}{
					1 &  0_{\R^{d}}\\
					1 &  0_{\R^{d}}\\
					}, \binom{-t_i}{-t_{i+1}}\right), \left( 
\bmat{cc|c|c}{N & -N & 0 & 0_{\R^{d}}}, 0\right)\right).
\end{align*}
We have that $W(\Phi^{\mathrm{(ind)}}_i) = 7$, $L(\Phi^{\mathrm{(ind)}}_i) = 3$ and, for $t\in [0,T]$ and $x \in \Omega$,
\begin{align}\label{eq:RealizationofInd}
\Realization\left(\Phi^{\mathrm{(ind)}}_i\right)(t,x) = N \cdot (\varrho(t- t_i) - \varrho(t- t_{i+1})) = \left\{ \begin{array}{ll}
0 & \text{ if } t \leq t_i,\\
N \cdot (t - t_i) & \text{ if } t_i < t < t_{i+1},\\ 
1 & \text{ if } t \geq t_{i+1}.\\
\end{array}\right.
\end{align}
Let $\bar{a} \coloneqq \|\Realization(\Phi)\|_{L^\infty([0,T]\times \Omega)}$. Now we set, for $i \in \{0, \dots, N-1\}$, 
\begin{align*}
\Phi^{\mathrm{(clip)}}_i\coloneqq \left( \left( \left(\begin{array}{cc} 2\bar{a} & 1\\2\bar{a} & 0 \end{array} \right), \binom{-\bar{a}}{-\bar{a}}\right), \left(\bmat{cc|cc}{1& -1 & 0 & 0}, 0 \right)\right) \sconc \mathrm{P}\left(\Phi^{\mathrm{(ind)}}_i, \Phi_i^{\mathrm{(shift)}}\right).
\end{align*}
We have that 
\begin{align}
\Realization\left(\Phi^{\mathrm{(clip)}}_i\right)(t,x) &= \varrho\left(2\bar{a}\Realization\left(\Phi_i^{\mathrm{(ind)}}\right)(t,x) + \Realization\left(\Phi_i^{\mathrm{(shift)}}\right)(t,x) -\bar{a}  \right) -  \varrho\left(2\bar{a} \Realization\left(\Phi_i^{\mathrm{(ind)}}\right)(t,x) -\bar{a}  \right).\label{eq:FormOfPhiClip}
\end{align}
It follows from \eqref{eq:RealizationofInd} and \eqref{eq:FormOfPhiClip} that, for $t \in [0,T]$ and $x \in \Omega$, 
\begin{align}
\Realization\left(\Phi^{\mathrm{(clip)}}_i\right)(t,x) &=  0, \text{ if } t \leq t_{i}\label{eq:clipProp1}\\
\Realization\left(\Phi^{\mathrm{(clip)}}_i\right)(t,x) &=  \Realization\left(\Phi_i^{\mathrm{(shift)}}\right)(t,x), \text{ if } t \geq t_{i+1},\label{eq:clipProp2}\\
\left|\Realization\left(\Phi^{\mathrm{(clip)}}_i\right)(t,x)\right| &\leq 2\bar{a}, \text{ else}.\label{eq:clipProp3}
\end{align}
 In addition, by Propositions \ref{prop:SparseConc} and \ref{prop:parallelization}, 
\begin{align}
L\left(\Phi^{\mathrm{(clip)}}_i\right) &= 2 + \max\{3, L(\Phi) + 2\}, \label{eq:estimateClipDepth}\\
W\left(\Phi^{\mathrm{(clip)}}_i\right) &\leq 16 + 2\cdot (7 + 2 W(\Phi) + 2d). \label{eq:estimateClipWidth}
\end{align}
Finally, we set 
\begin{align*}
\widetilde{I}_N(\Phi) \coloneqq \left( \left(\bmat{ccc}{\frac{1}{N} & \dots & \frac{1}{N}}, 0 \right) \right) \sconc \mathrm{P}\left(\Phi^{\mathrm{(clip)}}_0, \dots, \Phi^{\mathrm{(clip)}}_{N-1}\right).
\end{align*}
Now we have that $t \leq t_{i}$ if $\lceil t N / T \rceil \leq i$ and $t \geq t_{i+1}$ if $i \leq \lceil tN/T \rceil - 1$. Hence, for $t \in [0,T]$ and $x \in \Omega$, 
\begin{align*}
\Realization\left(\widetilde{I}_N(\Phi)\right)(t,x) &= \frac{1}{N} \sum_{i=0}^{N-1} \Realization\left(\Phi_i^{\mathrm{(clip)}}\right)(t,x)\\
&\underset{\footnotesize{\eqref{eq:clipProp1}}}{=} \frac{1}{N} \sum_{i=0}^{\lceil tN/T\rceil-1} \Realization\left(\Phi_i^{\mathrm{(clip)}}\right)(t,x)\\
&\underset{\footnotesize{\eqref{eq:clipProp2}}}{=}  \frac{1}{N} \sum_{i=0}^{\lceil tN/T\rceil-2} \Realization\left(\Phi_i^{\mathrm{(shift)}}\right)(t,x) + \frac{1}{N} \Realization\left(\Phi_{\lceil tN/T\rceil-1}^{\mathrm{(clip)}}\right)(t,x)\\
&= \frac{1}{N} \sum_{i=0}^{\lfloor tN/T\rfloor - 2} \Realization\left(\Phi\right)(t_i,x) + \frac{1}{N} \Realization\left(\Phi_{\lceil tN/T\rceil - 1}^{\mathrm{(clip)}}\right)(t,x)\\
&= \frac{1}{N} \sum_{i=0}^{\lceil tN/T\rceil-1} \Realization\left(\Phi\right)(t_i,x) + \frac{1}{N} \left(\Realization\left(\Phi_{\lceil tN/T\rceil-1}^{\mathrm{(clip)}}\right)(t,x) - \Realization\left(\Phi\right)\left(t_{\lceil tN/T\rceil-1},x\right)\right).
\end{align*}
Since, by \eqref{eq:clipProp3}, 
$$
\frac{1}{N} \left|\Realization\left(\Phi_{\lceil tN/T\rceil-1}^{\mathrm{(clip)}}\right)(t,x) - \Realization\left(\Phi\right)\left(t_{\lceil tN/T\rceil-1},x\right)\right| \leq 3\|\Realization(\Phi)\|_{L^\infty([0,T]\times \Omega)},
$$
we conclude the proof by observing with \eqref{eq:estimateClipDepth} and \eqref{eq:estimateClipWidth} that
\begin{align*}
L\left(\widetilde{I}_N(\Phi)\right) &\leq 3 + \max\{3, L(\Phi) + 2\},\\
W\left(\widetilde{I}_N(\Phi)\right) &\leq 2N + N\cdot (32 + 4\cdot (7 + 2 W(\Phi) + 2d)) = 62N + 8 W(\Phi) N + 8d N.
\end{align*}
\end{proof}

\begin{proposition}[Left Riemann sum]\label{prop:trapez}
Let $M>0$, $n \in \N$, $f \in C^1([0,T]\times [-M,M]^n; \R)$, and $N \in \N$. The approximation of the integral of $f$ with respect to its first argument from 0 to $t\leq T$ by the left Riemann sum is given by
\begin{align*}
    I_N(f)(t,x) \coloneqq \frac{1}{N} \sum_{i=0}^{N-1} f(t_i,x) \mathds{1}_{\{t_i< t\}}, \qquad t_i=i T /N.
\end{align*}
Then 
\begin{align*}
    \sup_{t\in [0,T], x \in [-M,M]^n}\left|\int_0^t f(\tau, x)\dtau -I_N(f)(t,x)\right| \leq \frac{2}{N} \|f\|_{C^1}.
\end{align*}
\end{proposition}
\begin{proof}
Let $N(t) \coloneqq \max\{i \in \N \,|\, t_i<t\}$. Then 
\begin{align*}
   \left|\int_0^t f(\tau, x)\dtau -I_N(f)(t,x)\right| &= \left|\sum_{i=0}^{N(t)} \int_{t_i}^{t_{i+1}} f(\tau) - f(t_i)  \dtau - \int_{{t}}^{{t_{N(t)+1}}} f(\tau) \dtau \right|\\
   &\leq \left|\sum_{i=0}^{N(t)} \int_{t_i}^{t_{i+1}} f(\tau) - f(t_i)  \dtau \right| + \left|\int_{t}^{t_{N(t)+1}} f(\tau) \dtau \right|\\
   &\leq \frac{1}{N} \mathrm{Lip_f} + \frac{1}{N} \|f\|_{C^0} \leq \frac{2}{N} \|f\|_{C^1}.
\end{align*}
\end{proof}
\remark \label{rmk:TrapRule}
Equation \ref{eq:ApproxStatementOfIntegralProposition} implies that for a NN $\Phi$ with $n+1$-dimensional input there holds
\begin{align*}
    \sup_{t\in [0,T], x \in [-M,M]^n}\left| \Realization\left(\widetilde{I}_N(\Phi)(t,x)\right) - I_N(\Realization(\Phi))(t,x)\right| \leq \frac{c}{N}
\end{align*}
with $c = 3\|\Realization(\Phi)\|_{L^\infty([0,T]\times \Omega)} > 0$.

\bibliographystyle{abbrv}
\bibliography{references}

\end{document}